\documentclass{article}
\pdfoutput=1
\usepackage[utf8]{inputenc}
\usepackage{amsmath}
\usepackage{amsthm}
\usepackage{amssymb}
\usepackage{tikz}
\usepackage{graphicx}
\usepackage{float}
\usepackage{comment}
\usepackage{subfigure}
\graphicspath{ {./images/} }
\theoremstyle{definition}

\newtheorem{lem}{Lemma}[section]
\newtheorem{proposition}{Proposition}[section]
\newtheorem{col}{Corollary}[section]
\newtheorem{theo}{Theorem}[section]

\newtheorem*{note}{Note}
\newtheorem*{pro}{Proof}
\def\checkmark{\tikz\fill[scale=0.4](0,.35) -- (.25,0) -- (1,.7) -- (.25,.15) -- cycle;} 

\usepackage{titlesec}

\titleformat{\paragraph}
{\normalfont\normalsize\bfseries}{\theparagraph}{1em}{}
\titlespacing*{\paragraph}
{0pt}{3.25ex plus 1ex minus .2ex}{1.5ex plus .2ex}

\usepackage[top=2.5cm, left=3cm, right=3cm, bottom=4.0cm]{geometry}
\author{Joel Foisy, Luis Ángel Topete Galván, Evan Knowles, Uriel Alejandro Nolasco,
\and Yuanyuan Shen, and Lucy Wickham}
\title{Intrinsically projectively linked graphs}

\begin{document}

\maketitle

\begin{abstract}
A graph is \textit{intrinsically projectively linked (IPL)} if its every embedding in projective space contains a nonsplit link. Some minor-minimal IPL graphs have been found previously \cite{2008REU}. We determine that no minor-minimal IPL graphs on 16 edges exists and identify new minor-minimal IPL graphs by applying $\Delta-Y$ exchanges to $K_{7}-2e$. We prove that for a nonouter-projective-planar graph $G$, $G+\bar{K}_{2}$ is IPL and describe the necessary and sufficient conditions on a projective planar graph $G$ such that $G+\bar{K}_{2}$ is IPL. Lastly, we deduce conditions for $f(G + \bar{K_{2}})$ to have no nonsplit link, where $G$ is projective planar, $\bar{K_{2}} = \{w_{0},w_{1}\}$, and $f(G + \bar{K_{2}})$ is the embedding onto $\mathbb{R}P^{3}$ with $f(G)$ in $z=0$, $w_{0}$ above $z=0$, and $w_{1}$ below $z=0$ such that every edge connecting ${w_{0},w_{1}}$ to $G$ avoids the boundary of the 3-ball, whose antipodal points are identified to obtain projective space.
\end{abstract}
\renewcommand{\thefootnote}{\fnsymbol{footnote}} 
    
\renewcommand{\thefootnote}{\arabic{footnote}} 
\section{Introduction}

Let $\mathbb{R}P^2$ denote real projective plane, which we take to be the unit disk in $\mathbb{R}^2$ with antipodal points on its boundary identified and let $\mathbb{R}P^3$ denote real projective space, which we take to be the closed unit ball in $\mathbb{R}^3$ with antipodal points on its boundary identified. In $\mathbb{R}P^3$, a two component link $L_{1} \cup L_{2}$ is \textit{split} if there exists a subset $A$ homeomorphic to a 3-ball, such that $L_{1} \subset A$, $L_{2} \subset A^{c}$. A graph $G$ is \textit {intrinsically linked in $\mathbb{R}P^{3}$} if in every embedding of $G$ in $\mathbb{R}P^{3}$ there is a nonsplit two component link. In 1993, Robertson, Seymour and Thomas characterized the set of all minor-minimal intrinsically linked graphs in $\mathbb{R}^{3}$: the Petersen family graphs \cite{robertson}. This family consists of the seven graphs obtained from $K_6$ through a sequence of $\Delta-Y $ and $Y- \Delta $ exchanges. We wish to obtain a similar result in $\mathbb{R}P^3$ and find the complete set of minor-minimal intrinsically projectively linked (IPL) graphs. All minor-minimal intrinsically linked in ${\mathbb R}P^3$  signed graphs have been classified \cite{Duong}.\\

It is known that exactly one graph in the Petersen family,  $K_ {4,4}-e$, is intrinsically projectively linked (IPL) \cite{2008REU}. The graph $K_6 \therefore K_6$, obtained by gluing two copies of $K_6$ along three vertices $v_1$, $v_2$, $v_3$ and then removing the triangle composed of the three edges between $v_1$, $v_2$, and $v_3$, is also IPL. There are two graphs obtained by deleting any two edges from $K_{7}$, either the two deleted edges are adjacent or nonadjacent, we denote either resulting graph as $K_{7}-2e$. In either case, we show the resulting graph is minor-minimal IPL. All other known minor-minimal IPL graphs have at least 9 vertices and 17 edges. We explore candidates for minor-minimal IPL graphs with 8 vertices and 16 edges by applying $\Delta-Y$ exchanges, that preserve the property of being IPL \cite{2008REU}, to $K_7-2e$ and by adding an edge or applying vertex splittings to the Peterson family graphs. We show there are no new IPL graphs on 16 edges and classifying IPL graphs on 8 vertices remains open.\\


 A cycle embedded in $\mathbb{R}P^3$ is \emph{0-homologous} if and only if it bounds a disk. This is also called a \textit{null cycle} \cite{Glover}. A cycle embedded in $\mathbb{R}P^3$ is \emph{1-homologous} if and only if it does not bound a disk. This is also called an \textit{essential cycle} \cite{Glover}. A projective planar graph $G$ is \textit{separating projective planar} if for every embedding of $G$ into $\mathbb{R}P^2$, there is a 0-homologous cycle $C$ in $G$ such that one vertex of $G\setminus C$ is in one connected component of ${\mathbb{R}P}^2\setminus C$ and another vertex of $G\setminus C$ is in the other connected component of $\mathbb{R}P^2\setminus C$. If there exists an embedding of graph $G$ that does not have this property, the graph is \textit{nonseparating projective planar}. A separating projective planar graph is nonouter-projective-planar. We prove that for any nonouter-projective-planar graph $G$, $G + \overline{K}_2$ is IPL by examining the 32 minor-minimal nonouter-projective-planar graphs and identifying subgraphs of known IPL graphs \cite{Archdeacon}. \\

There is no known $n$ for which $K_n$ has, in every embedding in projective space, a nonsplit link with at least one component 0-homologous. It is not even known if such an $n$ exists. Lastly, we deduce conditions for $f(G + \bar{K_{2}})$ to have no nonsplit link, where $G$ is projective planar, $\bar{K_{2}} = \{w_{0},w_{1}\}$, and $f(G + \bar{K_{2}})$ is the embedding onto $\mathbb{R}P^{3}$ with $f(G)$ in $z=0$, $w_{0}$ above $z=0$, and $w_{1}$ below $z=0$ such that every edge connecting ${w_{0},w_{1}}$ to $G$ avoids the boundary of the ball whose quotient gives projective space. We show $f(G + \bar{K_{2}})$ has no nonsplit link of two 0-homologous cycles if and only if $f(G)$ is a nonseparating projective planar embedding, $f(G + \bar{K_{2}})$ has no nonsplit link of a 0-homologous cycle and a 1-homologous cycle if and only if $f(G)$ has no separating 1-homologous cycle, and $f(G + \bar{K_{2}})$ has no nonsplit link of two 1-homologous cycles if and only if every 1-homologous edge share a common vertex in $f(G)$ or $f(G)$ has exactly three 1-homologous edges intersecting pairwise. \\

\section{Definitions and notation}

All of our graphs will be embedded piecewise linearly. Thus, for every graph, we may assume every cycle intersects the boundary at a finite number of points.\\

An embedding of graph in $\mathbb{R}P^3$ is an \emph{affine embedding} if the graph embedding does not intersect the boundary of the ball we use to define projective space.\\

A graph that can be obtained from a graph $G$ by a series of edge deletions, vertex deletions and edge contractions is called a \emph{minor} of $G$. The graph $G$ is \emph{minor-minimal} if, whenever $G$ has property $P$ and $H$ is a minor of $G$, then $H$ does not have property $P$. A property $P$ is \emph{minor closed} if, whenever a graph $G$ has property $P$ and $H$ is a minor of $G$, then $H$ also has property $P$. If $P$ is a minor closed property and the graph $G$ does not have property $P$, then $G$ is a \emph{forbidden graph} for $P$. Robertson and Seymour's Minor Theorem states if $P$ is a minor-closed graph property, then the minor-minimal forbidden graphs for $P$ form a finite set \cite{robertson}. \\

An \emph{outer-projective-planar graph} is one that can be embedded in the projective plane with all vertices in the same face - this is a minor closed property. Other minor closed properties include having a nonseparating projective planar embedding \cite{Dehkordi}, having a planar embedding and having a linkless embedding in space \cite{2008REU}.\\

Given two graphs $G$, $H$, their \emph{sum} $G+H$ is formed by starting with the disjoint union of $G$ and $H$, and adding an edge between every vertex in $G$ and every vertex in $H$. A \emph{complete graph} is a graph with an edge between all possible pairs of vertices in the graph. We represent a complete graph as $K_m$, where $m$ is the order of the graph. A \emph{complete bipartite graph} is a graph whose vertices can be divided into two disjoint sets, where no two vertices in the same set are adjacent and every vertex of the first set is adjacent to every vertex of the second set. We represent a complete bipartite graph as $K_{n,m}$, where the cardinalities of the two sets are $n$ and $m$. A \emph{complete $k$-partite graph} is a graph whose vertices can be divided into $k$ disjoint sets, where no two vertices in the same set are adjacent and every vertex of a given set is adjacent to every vertex in every other set. We represent a complete $k$-partite graph as $K_{n_1, \dots, n_k}$, where the cardinalities of the sets are $n_1, \dots, n_k$. \\

\section{Intrinsically projectively linked graphs}

\subsection{$\Delta$-Y, Y-$\Delta$ exchanges}

In the first part of this section, we prove that the $Y-\Delta$ exchange preserves nonseparating projective planarity and that the $\Delta-Y$ exchange preserves separating projective planarity. The following lemma follows easily from \cite{Glover}.

\begin{lem}
Nonprojective-planarity is not always preserved by Y-$\Delta$ exchanges. 
\end{lem}

\begin{proof}
The below graph $C_{11}$ is not projective planar \cite{Glover}. It has two components, $A$ and $B$.

\begin{figure}[H]
    \centering
    \includegraphics[scale=0.4]{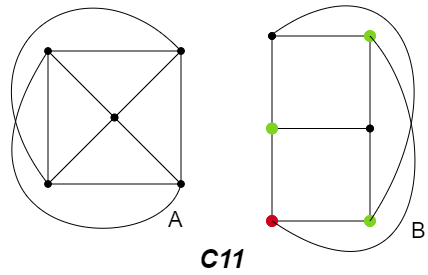}

    \label{fig:c11}
\end{figure}
Perform a Y-$\Delta$ exchange on the highlighted vertices of $B$, where the red vertex shares an edge with each of the three green vertices and is taken to be the center of the Y. Let $B'$ denote the resulting graph. Subgraph $A$ is projective planar and subgraph $B^{'}$ is planar. Therefore, the resulting graph $C_{11}^{'}$ is projective planar.
    \begin{figure}[ht]
\centering
\begin{minipage}[b]{0.4\linewidth}
			  \centering
    \includegraphics[scale=0.4]{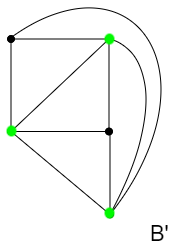}
    \caption{Subgraph $B^{'}$ from a Y-$\Delta$ on $B$}
\end{minipage}
\quad
\begin{minipage}[b]{0.4\linewidth}
			  \centering
    \includegraphics[scale=0.4]{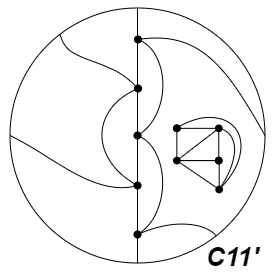}
    \caption{Embedding of $C_{11}^{'}$ in $\mathbb{R}P^2$}
\end{minipage}
\end{figure}

\end{proof}

\begin{col}
Projective planarity is not always preserved by $\Delta$-Y exchanges.
\end{col}

\begin{theo}
Nonseparating projective planarity is preserved by Y-$\Delta$ exchanges.
\end{theo}

\begin{proof}
Let $G$ have a nonseparating embedding with Y having vertices $v_{0}$, $v_{1}$, $v_{2}$ and $v_{3}$ and perform a Y-$\Delta$ exchange. By ambient isotopy, we may assume that the exchange takes place inside a sufficiently small neighbourhood $U$ of $v_{0}$ that contains only $v_{1}$, $v_{2}$, $v_{3}$. Since $U$ is bounded by a 0-homologous cycle, we may assume that $U$ is affine. Note that the Y-$\Delta$ exchange preserves projective planarity, and name the resulting embedding $G'$.\\

\begin{center}
    \includegraphics[scale=0.6]{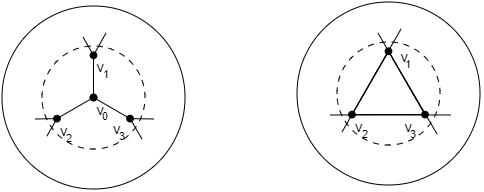}
\end{center}

Suppose to the contrary that $G'$ has a separating cycle $C'$. There are four cases to consider:

\begin{enumerate}
    \item $C'$ shares zero edges with ${(v_1, v_2), (v_1, v_3), (v_2, v_3)}$\\
    Then $U$ lies on one side of $C'$. Suppose that $C'$ separates $v_{i},v_{j}$ in $G'$, we must have $v_{i},v_{j}$ lies on two sides of $C'$ in $G$. Therefore $C'$ separates $v_{i},v_{j}$ in $G$, this is a contradiction.
    \item $C'$ shares one edge with ${(v_1, v_2), (v_1, v_3), (v_2, v_3)}$\\
    Without loss of generality suppose $C'$ consists of $L_{1}=(v_1, v_2)$ and $L_{2}=C'\setminus L_{1}$. Let $C$ be the cycle in $G$ consisting of $(v_0, v_1)$, $(v_0, v_2)$ and $L_{2}$. Suppose $C'$ separates $v_{i},v_{j}$ in $G'$, then $C$ separates $v_{i},v_{j}$ in $G$, which is a contradiction.
    \item $C'$ shares two edges with ${(v_1, v_2), (v_1, v_3), (v_2, v_3)}$\\
    Without loss of generality suppose $C'$ consists of $L_{1}=v_{1} \cup (v_1, v_2) \cup (v_1, v_3)$ and $L_{2}=C'\setminus L_{1}$. Let $C$ be the cycle in $G$ consisting of $(v_0, v_2)$, $(v_0, v_3)$ and $L_{2}$. Suppose $C'$ separates $v_{i},v_{j}$ in $G'$, then $v_{i},v_{j}$ lie on the other side of $U$. Therefore $C$ separates $v_{i},v_{j}$ in $G$, which is a contradiction.
    \item $C'$ shares 3 edge with ${(v_1, v_2), (v_1, v_3), (v_2, v_3)}$\\
    Then $C'$ is the cycle consisting of ${(v_1, v_2), (v_1, v_3), (v_2, v_3)}$. However, by assumption $U$ contains no other vertex, hence $C'$ cannot be a separating cycle.
\end{enumerate}
\end{proof}

\begin{col}
Given a particular $\Delta$-Y exchange that preserves projective planarity of $G$, this $\Delta$-Y exchange preserves separating projective planarity on G.
\end{col}
\begin{proof}
Suppose $G$ is a separating projective planar graph. Name the graph after the $\Delta$-Y exchange to be $G'$. Since the Y-$\Delta$ exchanges preserve projective planarity, every embedding of $G'$ is obtained from a separating embedding of $G$ via a $\Delta$-Y exchange on the projective plane. The result then follows.
\end{proof}

\subsection{$K_{7}-2e$}

To find a new minor-minimal IPL graph, we performed $\Delta - Y$ exchanges on $K_{7}-2e$ in attempt to find a graph that is IPL that does not contain $K_{7}-2e$ or $K_{4,4}-e$ as a minor.\\

Up to symmetry there are two cases of removing two edges from $K_{7}$: the two edges being adjacent (Case 1) or nonadjacent (Case 2). Consider performing a $\Delta - Y$ exchange on a triangle in $K_{7}-2e$. By symmetry there are three cases: the triangle has no vertex with its edge removed from $K_{7}$, has 1 vertex with its edge removed, or has 2 vertices with their edges removed.

\begin{figure}[H]
\begin{center}
    \includegraphics[scale=0.6]{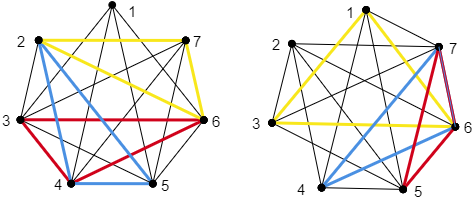}
    \caption{Two cases of removing two edges from $K_{7}$}
\end{center}
\end{figure}

\begin{theo} The below table summarizes the results, where a \checkmark  indicates that a new minor-minimal intrinsically linked graph exists.
\begin{center}
\begin{tabular}{ |c|c|c|c|} 
\hline
\# vertices with edges removed & 0 & 1 & 2\\ 
\hline
Case 1 & & \checkmark \checkmark & \\ 
\hline
Case 2 & \checkmark & & \checkmark \\ 
\hline
\end{tabular}
\end{center}
\end{theo}

\begin{proof}
Note that a graph resulting from performing a $\Delta - Y$ exchange on a $K_{7} - 2e$ does not contain a $K_{7} - 2e$ as a minor. Therefore it suffices to check whether the resulting graph contains $K_{4,4}-e$, the only known IPL graph on 8 vertices, as a minor. Take the two vertex sets of $K_{4,4}-e$ to be $V_{0}$, $V_{1}$ and separately consider different cases.
\subsubsection{Case 1}

In Case 1, the two edges removed from $K_{7}$ are adjacent. The $\Delta - Y$ exchange can be performed on four types of triangles in $K_{7} - 2e$.

\begin{figure}[H]
    \centering
    \subfigure[]{\includegraphics[width=0.24\textwidth]{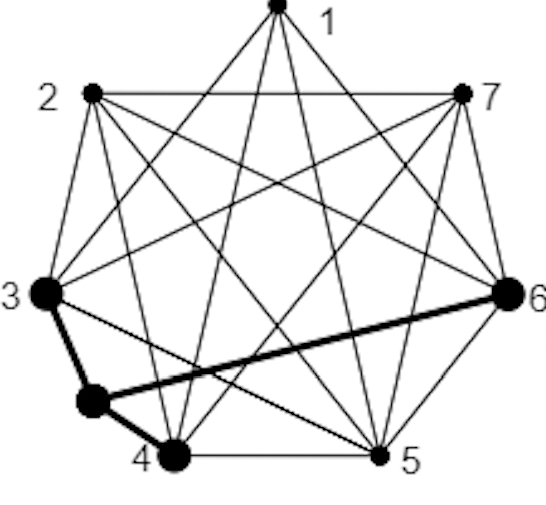}} 
    \subfigure[]{\includegraphics[width=0.24\textwidth]{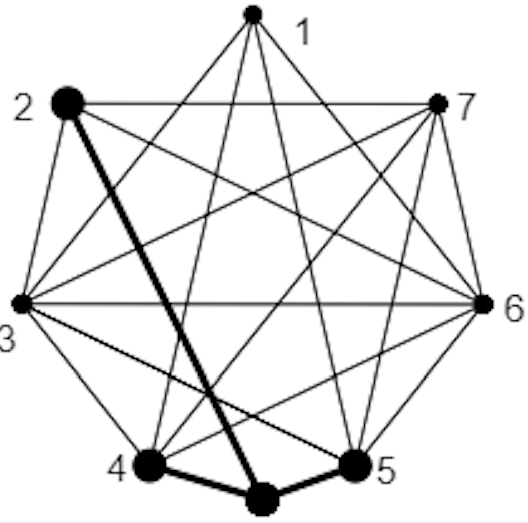}} 
    \subfigure[]{\includegraphics[width=0.24\textwidth]{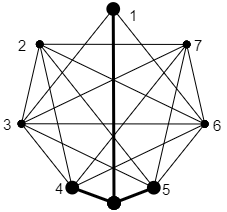}}
    \subfigure[]{\includegraphics[width=0.24\textwidth]{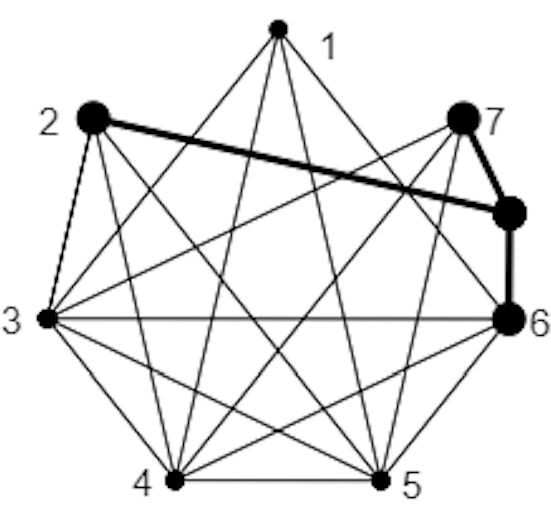}}
    \caption{(a) The $\Delta - Y$ exchange performed on a triangle that has zero vertices with an edge removed (b) The $\Delta - Y$ exchange performed on a triangle that has one vertex with an edge removed  (c) The $\Delta - Y$ exchange performed on a triangle that has one vertex with an edge removed  (d) The $\Delta - Y$ exchange performed on a triangle that has two vertices with an edge removed }
    \label{fig:foobar}
\end{figure}


\begin{itemize}
    \item Triangle has no vertex with an edge removed, as in Figure \ref{fig:foobar} (a).\\
    
    Without loss of generality suppose the triangle has vertices ${v_{3}, v_{4}, v_{6}}$. Consider $V_{0} = \{ v_{1}, v_{2}, v_{7}, v_{8} \}$ and $V_{1} = \{v_{3}, v_{4}, v_{5}, v_{6} \}$, then the resulting graph contains a minor of $K_{4,4}-e$.\\
    
    \item Triangle has exactly one vertex $v$ with an edge removed. There are two sub-cases as in Figure \ref{fig:foobar} (b), (c).\\
    
    Note that $v_{8}$ in the resulting graph has degree three, and corresponds to a degree three vertex in $K_{4,4}-e$ if the graph contains a minor of $K_{4,4}-e$. Therefore if $v_{i}$, $v_{j}$ are not connected by an edge, where $1 \leq i,j \leq 7$, they are in different vertex sets. 
    
    \begin{itemize}
        \item $v$ has one edge removed\\
        
        Without loss of generality suppose the triangle has vertices $v_{2}, v_{4}, v_{5}$. Since no edge connects $v_{7}$ with $v_{1}$, $v_{1}$ with $v_{2}$, $v_{2}$ with $v_{4}$, and $v_{4}$ with $v_{5}$, then $v_{7}$, $v_{1}$, $v_{2}$, $v_{4}$, $v_{5}$ must be in the same vertex set. This is a contradiction since a vertex set has four vertices. There is a new IPL graph in this case.
        
        \item $v$ has two edges removed\\
        
        Without loss of generality suppose the triangle has vertices $v_{1}, v_{4}, v_{5}$. 
        Since no edge connects $v_{1}$ with $v_{2}$, $v_{1}$ with $v_{4}$, $v_{1}$ with $v_{5}$, and $v_{1}$ with $v_{7}$, then $v_{1}$, $v_{2}$, $v_{4}$, $v_{5}$, $v_{7}$ must be in the same vertex set. This is a contradiction since a vertex set has four vertices. There is a new IPL graph in this case.
    \end{itemize}
    
    \item Triangle has two vertices with an edge removed, as in Figure \ref{fig:foobar} (d).\\
    
    Without loss of generality suppose the triangle has vertices ${v_{2}, v_{6}, v_{7}}$. Consider $V_{0} = \{v_{1}, v_{2}, v_{6}, v_{7}\}$ and $V_{1} = \{v_{3}, v_{4}, v_{5}, v_{8}\}$, then the resulting graph contains a minor of $K_{4,4}-e$.\\
\end{itemize}

\subsubsection{Case 2}

In Case 2, the two edges removed from $K_{7}$ are nonadjacent.

\begin{figure}[H]
    \centering
    \subfigure[]{\includegraphics[width=0.25\textwidth]{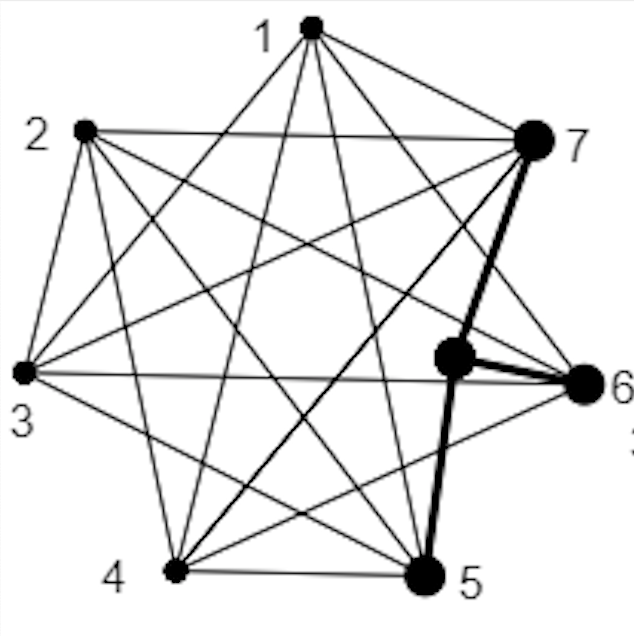}} 
    \subfigure[]{\includegraphics[width=0.25\textwidth]{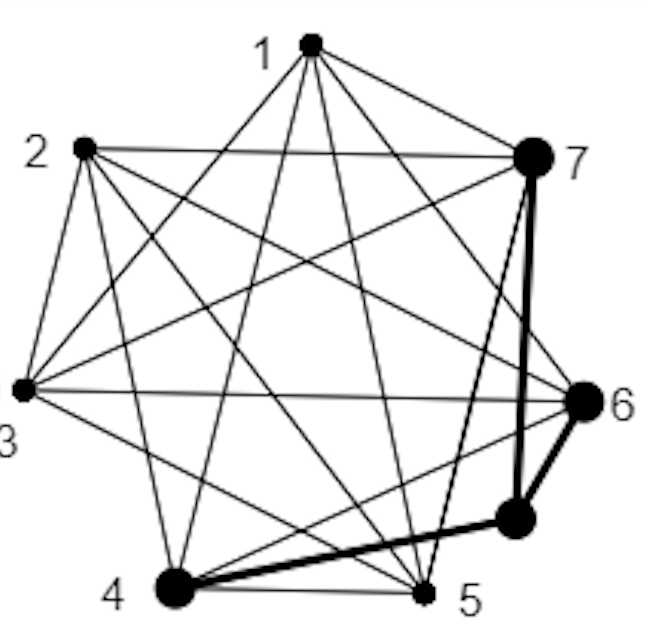}} 
    \subfigure[]{\includegraphics[width=0.25\textwidth]{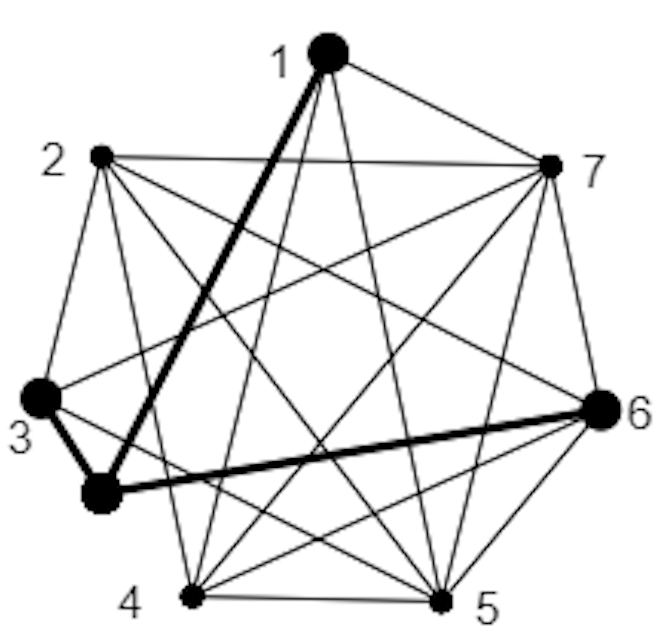}}
    \caption{(a) The $\Delta - Y$ exchange performed on a triangle that has zero vertices with an edge removed (b) The $\Delta - Y$ exchange performed on a triangle that has one vertex with an edge removed  (c) The $\Delta - Y$ exchange performed on a triangle that has two vertices with an edge removed }
    \label{fig:foobar2}
\end{figure}

    
\begin{itemize}
        \item Triangle has zero vertices with an edge removed, as in Figure \ref{fig:foobar2} (a).\\
        
        Without loss of generality suppose the triangle has vertices $v_{5}, v_{6}, v_{7}$. Since no edge connects $v_{5}$ with $v_{6}$ and $v_{6}$ with $v_{7}$, $v_{5}$, $v_{6}$, $v_{7}$ are in the same vertex set. Similarly $v_{1}$ and $v_{2}$, $v_{3}$ and $v_{4}$ are in the same vertex set. Since a vertex set has four vertices, it must be that $V_{0} = \{ v_{1}, v_{2}, v_{3}, v_{4} \}$ and $V_{1} = \{ v_{5}, v_{6}, v_{7}, v_{8} \}$. However, $v_{8}$ does not connect to any vertex in $V_{0}$. Therefore the graph does not contain a minor of $K_{4,4}-e$. There is a new IPL graph in this case.\\
        
        \item Triangle has one vertex with an edge removed, as in Figure \ref{fig:foobar2} (b).\\
        
        Without loss of generality suppose the triangle has vertices ${v_{4}, v_{6}, v_{7}}$. Consider $V_{0} = \{ v_{1}, v_{2}, v_{5}, v_{8} \}$ and $V_{1} = \{ v_{3}, v_{4}, v_{6}, v_{7} \}$, then the resulting graph contains a minor of $K_{4,4}-e$.\\
        
        \item Triangle has two vertices with an edge removed, as in Figure \ref{fig:foobar2} (c).\\
        
        Without loss of generality suppose the triangle has vertices $v_{1}, v_{3}, v_{6}$. Since no edge connects $v_{1}$ with $v_{2}$, $v_{1}$ with $v_{3}$, $v_{1}$ with $v_{6}$, and $v_{3}$ with $v_{4}$, then $v_{1}$, $v_{2}$, $v_{3}$, $v_{4}$, $v_{6}$ must be in the same vertex set. This is a contradiction since a vertex set has four vertices. There is a new IPL graph in this case.

    \end{itemize}

\end{proof}





\subsection{Irreducible graphs}
A graph $G$ is \textit{irreducible} for a surface $M$ if $G$ does not embed into $M$ and every proper subgraph of $G$ embeds into $M$. There are 103 irreducible graphs for the projective plane \cite{Glover}.


\begin{theo}
If a graph $G$ is irreducible for the projective plane, then it has an embedding in projective space with no a nonsplittable 2-link containing at least one 0-homologous component.
\end{theo}
\begin{pro}
Since $G$ is irreducible, there exists a projective planar embedding of $G - e$. Suppose this embedding lies on $z=0$ and add $e$ onto this embedding in projective space. This is the only edge not on $z=0$ and it lies above all other edges, hence this embedding does not have a link with a 0-homologous cycle.
\end{pro}
\subsection{$G + \overline{K_2}$}

\begin{lem} [Bustamante et al \cite{2008REU}]
Let G=($P_1\cup P_2$)$\backslash$ $\{(v_1,v_2)\}$ be a graph, where $P_1,P_2$ are graphs in the Petersen family and $V(P_1\cap P_2)=\{v_1,v_2\}.$ Then $G$ is IPL. 
\end{lem}\label{Bus1}

If a connected graph $G$ is \emph{$k$-connected}, then $G$ must have more than $k$ vertices and it remains connected whenever fewer than $k$ vertices are removed. The \emph{$k$-connectivity} of a graph is the maximum $k$ for which the graph is \emph{$k$-connected}.

\begin{lem} [Bustamante et al \cite{2008REU}]
There are 18 IPL graphs with 3-connectivity that can be obtained from $K_6 \therefore K_6$ by $\Delta$-$Y$ exchanges.
\end{lem}\label{Bus2}

\begin{proposition}
Let $G$ be a nonouter-projective-planar graph. Then $G + \overline{K_2}$ is IPL.
\end{proposition}

\begin{proof}
Given G nonouter-projective-planar graph, it must contain one of the 32 minor-minimal nonouter-projective-planar graph as a minor \cite{Archdeacon}, so it is sufficient to prove these special cases. In the following graphs, the idea is to prove that $K_{4,4}-e$ is a minor, so red vertices will indicate the vertices that belong to a set of $K_{4,4}-e$ and green vertices belong to the other set. The vertices will be labeled with a number that indicates the number of edges that a vertex shares with the vertices of the other set. In addition, vertices with no edges drawn are the vertices of $\overline{K_2}$ and they are supposed to be connected to the rest of the vertices of $G$, but their edges have not been drawn in order to have a more understandable drawing.

\begin{figure}[H]
\centering
\begin{minipage}[b]{0.3\linewidth}
    \includegraphics[scale=0.55]{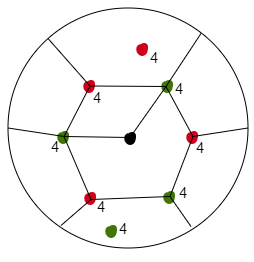}
    \caption{$\kappa_1+\overline{K_2}$}
\end{minipage}
\quad
\begin{minipage}[b]{0.3\linewidth}
    \includegraphics[scale=0.40]{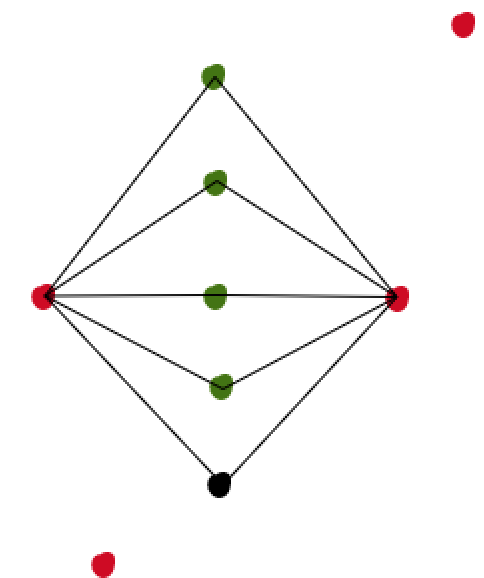}
    \caption{$\theta_1+\overline{K_2}$}
\end{minipage}
\quad
\begin{minipage}[b]{0.3\linewidth}
    \includegraphics[scale=0.55]{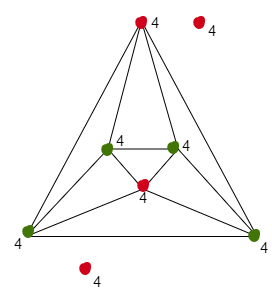}
    \caption{$\zeta_1+\overline{K_2}$}
\end{minipage}

\end{figure}

\begin{figure}[H]
\centering
\begin{minipage}[b]{0.3\linewidth}
    \includegraphics[scale=0.55]{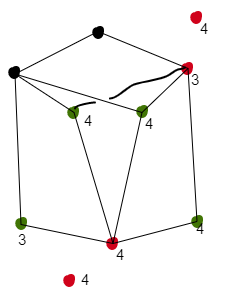}
    \caption{$\eta_1+\overline{K_2}$}
\end{minipage}
\quad
\begin{minipage}[b]{0.3\linewidth}
    \includegraphics[scale=0.55]{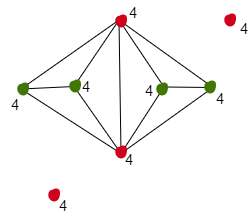}
    \caption{$\gamma_1+\overline{K_2}$}
\end{minipage}
\quad
\begin{minipage}[b]{0.3\linewidth}
    \includegraphics[scale=0.55]{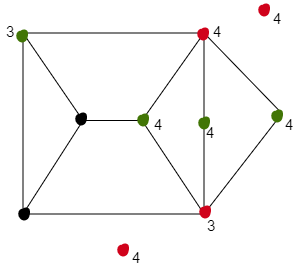}
    \caption{$\delta_1+\overline{K_2}$}
\end{minipage}
\end{figure}

\begin{figure}[H]

    \begin{center}
    \includegraphics[scale=0.65]{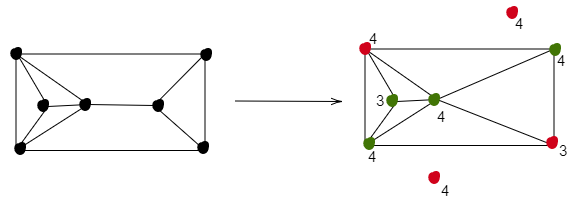}
    \caption{$\epsilon_1+\overline{K_2}$}
\end{center}

\end{figure}

For the cases of $\alpha_1$ and $\beta_2$ graphs, we use lemmas mentioned at the beginning of this subsection. Let $G_1=(K_6\cup K_6) \backslash \{(e,f)\}$ be a graph, where $V(K_6\cap K_6)=\{e,f\},$ and $G_2=K_6 \therefore K_6.$

\begin{figure}[H]
    \begin{center}
    \includegraphics[scale=0.75]{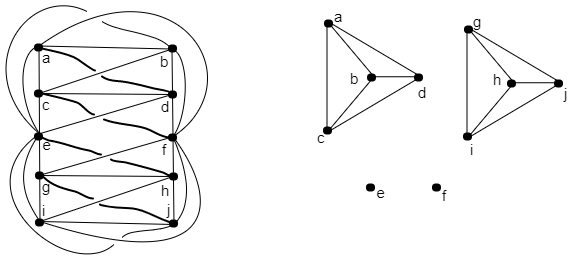}
    \caption{$G_1$ and $\alpha_1+\overline{K_2}$}
\end{center}

\end{figure}
\begin{figure}[H]

\begin{center}
    \includegraphics[scale=0.75]{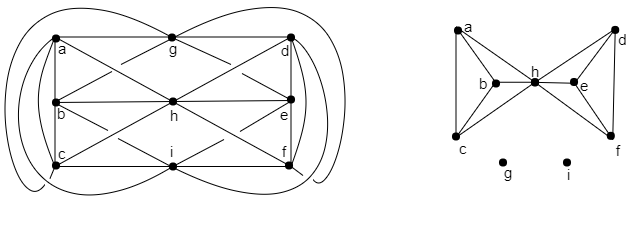}
    \caption{$G_2$ and $\beta_1+\overline{K_2}$}
\end{center}

\end{figure}

By the labeling, we see that $G_1$ is a minor of $\alpha_1+\overline{K_2}$ and $G_2$ is a minor of $\beta_1+\overline{K_2},$ so $\alpha_1+\overline{K_2}$ and $\beta_1+\overline{K_2}$ are IPL. Thus, by using $\Delta$-$Y$ exchanges, the rest of the graphs in the families will also be IPL.

\end{proof}




\subsection{Conditions for $G + \bar{K_{2}}$ to have no nonsplit links of different types}

It is an open question to determine the smallest $n$ such that $K_{n}$ has a nonsplit link with a 0 homologous cycle in every embedding into projective space. In this section we discuss conditions on an embedding of $G + \bar{K_{2}}$ to have no nonsplit links of a certain homology type.\\

Throughout this section, consider $G$, a connected graph with a projective planar embedding $f(G)$. Let $\bar{K_{2}}$ have the vertex set $\{w_{0},w_{1}\}$. Take  $f(G + \bar{K_{2}})$ to be the embedding onto $\mathbb{R}P^{3}$ with $f(G)$ in $z=0$, $w_{0}$ above $z=0$, and $w_{1}$ below $z=0$ such that every edge connecting ${w_{0},w_{1}}$ to $G$ does not cross the boundary and is straight. Let an edge be \textit{0-homologous} if it crosses the boundary an even number of times, and let an edge be \textit{1-homologous} if it crosses the boundary an odd number of times. \\

We deduce the necessary and sufficient conditions for $f(G + \bar{K_{2}})$ to have no nonsplit link by proving $f(G + \bar{K_{2}})$ has no nonsplit link of two 0-homologous cycles if and only if $f(G)$ is a nonseparating projective planar embedding, $f(G + \bar{K_{2}})$ has no nonsplit link of a 0-homologous cycle and a 1-homologous cycle if and only if $f(G)$ has no separating 1-homologous cycle or a 1-homologous cycle with vertices on both sides, and $f(G + \bar{K_{2}})$ has no nonsplit link of two 1-homologous cycles if and only if every 1-homologous edge share a common vertex in $f(G)$ or $f(G)$ has exactly three 1-homologous edges intersecting pairwise.

\begin{figure}[H]
    \begin{center}
    \includegraphics[scale=0.6]{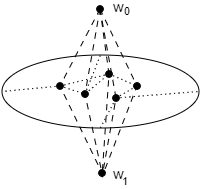}
    \caption{$f(G + \bar{K_{2}})$}
\end{center}

\end{figure}

Since ambient isotopy preserves split and nonsplit links, assume that no vertex of $f(G + \bar{K_{2}})$ lies on the boundary. Assume also that in $f(G + \bar{K_{2}})$, no loop exists and at most one edge connects any two vertices. 

\subsubsection{Nonsplit links between two 0-homologous cycles}

The properties of separating and nonseparating projective planar graphs are very helpful in characterizing links in projective space. The theorems in this section reveal more about this relationship. We have written a related paper where we characterized many members of the set of minor-minimal separating projective planar graphs \cite{REU2021planar}.

\begin{theo}
\label{Theorem1}
The embedding $f(G + \bar{K_{2}})$ has no nonsplit link of two 0-homologous cycles if and only if $f(G)$ is a nonseparating projective planar embedding.
\end{theo}
\begin{pro}

    Suppose there is a nonsplit link $L$ consisting of two components $C_{1}$, $C_{2}$. Note that if $w_{0}$, $w_{1}$ are not in the same component of $L$, $L$ must be split. Without loss of generality assume that $C_{1}$ has the form $(w_{0}, a_{1}, ... , a_{n}, w_{1}, b_{1}, ... b_{m})$, where $n, m \in \mathbb{Z^{+}}$ and $C_{2}$ lies on $z=0$. Note that all $a_{i}$ lie on the same side of $C_{2}$ on $z=0$, and all $b_{i}$ lie on the same side of $C_{2}$ on $z=0$. \\
    
    If $a_{i}$, $b_{i}$ lie on the same side of $C_{2}$, $L$ is split. If $a_{i}$, $b_{i}$ lie on different sides of $C_{2}$, then $C_{2}$ is a separating cycle in the projective planar embedding of $G$. In either case we have a contradiction. 

    Conversely, suppose $f(G)$ is a separating projective planar embedding, then there exists a cycle $C_{0}$ separating $v_{0}$, $v_{1}$. Consider the link $L$ consisting of $C_{0}$ and $C_{1} = (v_{0},w_{0},v_{1},w_{1})$, then $L$ is a nonsplit link of two 0-homologous cycles, contradicting the assumption.
\end{pro}
\begin{col}
If $G$ is a nonseparating projective planar graph, then $G+\bar{K_{2}}$ has an embedding with no link consisting of two 0-homologous cycles. 
\end{col}

\subsubsection{Nonsplit links between one 0-homologous cycle and one 1-homologous cycle}

Consider a projective planar embedding $f(G)$ of $G$ and a 0-homologous cycle bounding $f(G') \subset f(G)$ which contains every vertex of $f(G)$ as in Figure \ref{fig:1-homologous}. Then each 1-homologous cycle divides $f(G')$ into two sides. A 1-homologous cycle is \textit{separating} if there is at least one vertex on each of its two sides.

\begin{figure}[H]
    \begin{center}
    \includegraphics[scale=0.57]{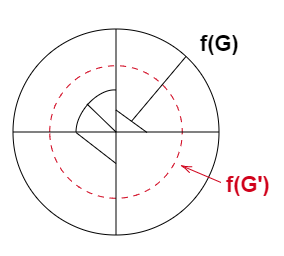}

    \includegraphics[scale=0.46]{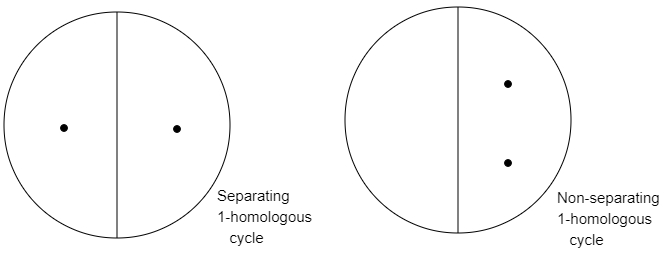}
    \caption{Separating 1-homologous Cycle}
    \label{fig:1-homologous}
\end{center}

\end{figure}

\begin{theo}
\label{Theorem2}
Given that $f(G)$ is nonseparating, $f(G+\bar{K_{2}})$ has no nonsplit link of a 0-homologous cycle and a 1-homologous cycle if and only if $f(G)$ has no separating 1-homologous cycle\\
\end{theo}
\begin{pro}
The two lemmas below complete the proof.
\end{pro}

\begin{lem}
Suppose $f(G+\bar{K_{2}})$ has a nonsplit link of a 0-homologous cycle and a 1-homologous cycle, then one of the following two cases occurs
\begin{enumerate}
    \item $f(G)$ has a separating 1-homologous cycle
    \item $f(G)$ is a separating projective planar embedding
\end{enumerate}
\end{lem}
\begin{pro}
Suppose there is a nonsplit link $L$ consisting of two components $C_{0}$, $C_{1}$, where $C_{0}$ is 0-homologous and $C_{1}$ is 1-homologous. Note that if $w_{0}$, $w_{1}$ are not in the same component of $L$, then $L$ must be split. Then one of the following two cases occurs
\begin{enumerate}
    \item $w_{0}$, $w_{1}$ in $C_{0}$\\
    
    Without loss of generality assume that $C_{0}$ has the form $(w_{0}, a_{1}, ... , a_{n}, w_{1},$ $b_{1}, ... b_{m})$, where $n, m \in \mathbb{Z^{+}}$. Note that all $a_{i}$ lie on the same side of $C_{1}$ on $z=0$, and all $b_{i}$ lie on the same side of $C_{1}$ on $z=0$. If $a_{i}$, $b_{i}$ lie on the same side of $C_{1}$, $L$ is split. If $a_{i}$, $b_{i}$ lie on different sides of $C_{1}$, then $C_{1}$ is a separating 1-homologous cycle.\\
    
    \item $w_{0}$, $w_{1}$ in $C_{1}$\\
    
    Without loss of generality assume that $C_{1}$ has the form $(w_{0}, a_{1}, ... , a_{n}, w_{1},$ $b_{1}, ... b_{m})$, where $n, m \in \mathbb{Z^{+}}$. Note that all $a_{i}$ lie on the same side of $C_{1}$ on $z=0$, and all $b_{i}$ lie on the same side of $C_{1}$ on $z=0$. If $a_{i}$, $b_{i}$ lie on the same side of $C_{1}$, $L$ is split. If $a_{i}$, $b_{i}$ lie on different sides of $C_{1}$, then $C_{0}$ is a separating 0-homologous cycle, and $f(G)$ is a separating projective planar embedding.\\
\end{enumerate}
\end{pro}

\begin{lem}
Suppose $f(G)$ has a separating 1-homologous cycle, then $f(G+\bar{K_{2}})$ has a nonsplit link of a 0-homologous cycle and a 1-homologous cycle.
\end{lem}

\begin{pro}
Consider a 1-homologous cycle $C_{1}$ separating $v_{0}$, $v_{1}$. Consider the link $L$ consisting of $C_{1}$ and $C_{0} = (v_{0}, w_{0}, v_{1}, w_{1})$, then $L$ is a nonsplit link of the 0-homologous cycle $C_{0}$ and the 1-homologous cycle $C_{1}$.
\end{pro}

\subsubsection{Nonsplit links between two 1-homologous cycles}

\begin{theo}
\label{Theorem3}
The embedding $f(G+\bar{K_{2}})$ has no nonsplit link of two 1-homologous cycles if and only if every pair of 1-homologous edges share a common vertex in $f(G)$, if and only if exactly one of the two following occurs
\begin{enumerate}
    \item Every 1-homologous edge shares a common vertex $v_{0}$ in $f(G)$.
    \item The embedding $f(G)$ has exactly three 1-homologous edges intersecting pairwise.
\end{enumerate}
\end{theo}

\begin{figure}[H]
    \begin{center}
    \includegraphics[scale=0.5]{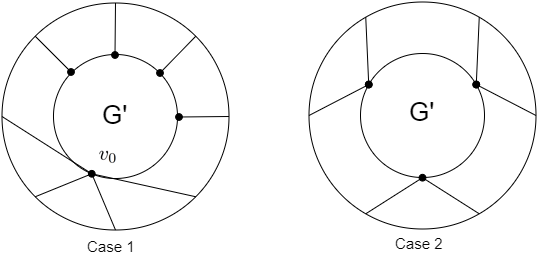}
    \caption{Case 1 and Case 2, where $f(G')$ is planar}
\end{center}

\end{figure}

\begin{proof}
We separately prove each statement of the theorem.\\

\begin{itemize}
    \item Every pair of 1-homologous edges share a common vertex implies no link of two 1-homologous cycles.
\end{itemize}

Every 1-homologous cycle contains a 1-homologous edge in $f(G)$. Therefore every two 1-homologous cycle intersects, and there is no link of two 1-homologous cycles. 

\begin{itemize}
    \item A pair of 1-homologous edges does not share a common vertex implies a link of two 1-homologous cycles exists.
\end{itemize}

Consider the edges between $v_{0}$, $v_{1}$ and $u_{0}$, $u_{1}$ to be the pair of 1-homologous edges. Then the cycles $C_{0}$ consisting of $(v_{0}, v_{1})$, $(v_{1}, w_{0})$, $(w_{0}, v_{0})$ and $C_{1}$ consisting of $(u_{0}, u_{1})$, $(u_{1}, w_{1})$, $(w_{1}, u_{0})$ form the desired link.

\begin{itemize}
    \item Case 1 or Case 2 implies every pair of 1-homologous edges share a common vertex in $f(G)$.
\end{itemize}

Note that in either case, every pair of 1-homologous edges share a common vertex in $f(G)$. 

\begin{itemize}
    \item  Every pair of 1-homologous edges share a common vertex in $f(G)$ implies Case 1 or Case 2.
\end{itemize}

Suppose not every 1-homologous edge passes through a common vertex, then there exists three 1-homologous edges that do not pass through a common vertex but intersect pairwise at $v_{0}$, $v_{1}$, $v_{2}$. Suppose there is a fourth 1-homologous edge, then one of the following three cases must occur

\begin{enumerate}
    \item This edge connects two of $v_{0}$, $v_{1}$, $v_{2}$.\\
    
    Without loss of generality suppose it connects $v_{0}$, $v_{1}$. Then there are two 1-homologous edges between $v_{0}$, $v_{1}$, contradicting multiplicity.
    \item This edge connects one of $v_{0}$, $v_{1}$, $v_{2}$ to a new vertex $u_{0}$.\\
    
    Without loss of generality suppose it connects $v_{0}$, $u_{0}$. Then the pair of 1-homologous edges of $(v_{0}, u_{0})$ and $(v_{1}, v_{2})$ does not share a common vertex, contradicting the assumption.
    \item This edge connects two new vertices $u_{0}$, $u_{1}$.\\
    
    The pair of 1-homologous edges of $(v_{0}, v_{1})$ and $(u_{0}, u_{1})$ does not share a common vertex, contradicting the assumption.
\end{enumerate}
Therefore there cannot be a fourth 1-homologous edge, and Case 2 must occur.

\end{proof}

\begin{col}
If $f(G+\bar{K_{2}})$ has no nonsplit link of two 1-homologous cycles, then $G$ is a planar graph.
\end{col}
\begin{proof}
Consider Case 1 and Case 2 separately
\begin{itemize}
    \item Case 1: Every 1-homologous edge shares a common vertex $v_{0}$ in $f(G)$
    
\begin{figure}[H]
    \begin{center}
    \includegraphics[scale=0.5]{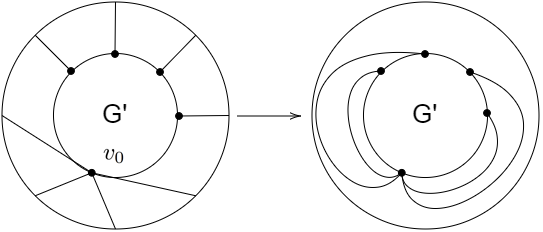}
    \caption{$G$ is planar}
    \label{fig:G-planar}
\end{center}

\end{figure}

    Suppose every 1-homologous edge is in the form $(v_{0},u_{i})$, $i \in \mathbb{Z}^{+}$ and let the set of 1-homologous edges be $V_{1}$. We claim $f(G)$ has an alternative embedding $f'(G)$ where $f(G \setminus V_{1})$ is unchanged and each edge incident to $v_{0}$ does not cross the boundary. Construct $f'(G)$ by inductively adding edges incident to $v_{0}$ not crossing the boundary to $f(G \setminus V_{1})$ as in Figure \ref{fig:G-planar}.\\
    
    Consider an affine embedding of $f(G \setminus V_{1})$ and let $F$ be the face containing the boundary of projective space. Note that $v_{0}$, $u_{i}$ all lie on the boundary of $F$. Since $v_{0}$, $u_{1}$ both lie on the affine boundary of a common face $F$, an edge between $v_{0}$, $u_{1}$ not crossing the boundary can be added. Either $(v_{0}, u_{1})$ divides $F$ into two faces, each having an affine boundary and having $v_{0}$ lying on its boundary, or $(v_{0}, u_{1})$ does not divide $F$ into two faces, with $F$ having an affine boundary and having $v_{0}$ lying on its boundary. Inductively note that each pair $v_{0}$, $u_{i}$ lie on the affine boundary of a common face, therefore an edge connecting $v_{0}$, $u_{i}$ not crossing the boundary can be added. The new edge divides the common face into two faces, each having an an affine boundary and having $v_{0}$ lying on its boundary.\\
    
    The resulting embedding is planar.
    
    \item Case 2: The embedding $f(G)$ has exactly three 1-homologous edges intersecting pairwise\\
    
    Consider the alternative planar embedding of $G$ illustrated below. 
    
\begin{figure}[H]
    \begin{center}
    \includegraphics[scale=0.5]{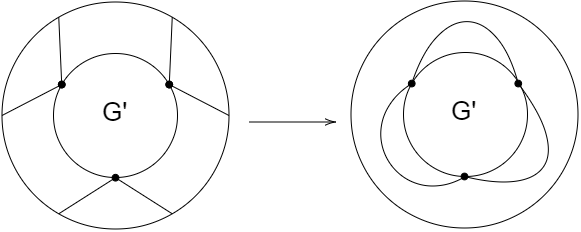}
    \caption{$G$ is planar}
\end{center}
\end{figure}
\end{itemize}
\end{proof}

By considering Theorem \ref{Theorem1}, Theorem \ref{Theorem2}, Theorem \ref{Theorem3}, we attempt to deduce all possible forms of $f(G)$ such that $f(G+\bar{K_{2}})$ has no nonsplit link. 
\subsubsection{Case 1}

Assume $f(G)$ has the embedding as in Case 1, then $f(G+\bar{K_{2}})$ has no link between two 1-homologous cycles. It suffices to look at links between two 0-homologous cycles and between one 0-homologous and one 1-homologous cycle.\\

\begin{theo}
Assume $f(G)$ has the embedding as in Case 1. Let $V_{0}$ be the set of 0-homologous edges connecting to $v_{0}$, and $V_{1}$ be the set of 1-homologous edges connecting to $v_{0}$. The embedding $f(G+\bar{K_{2}})$ has no nonsplit link of two 0-homologous cycles if and only if $f(G \setminus V_{0})$ and $f(G  \setminus V_{1})$ are equivalent to nonseparating planar embeddings.
\end{theo}
\begin{pro}

Note that $f(G \setminus V_{1})$ is an affine embedding, and $f(G  \setminus V_{0})$ can be isotoped to an affine embedding. Suppose $f(G \setminus V_{0})$ or $f(G \setminus V_{1})$ has a separating cycle, then $f(G)$ has a separating cycle. By Theorem \ref{Theorem1}, $f(G+\bar{K_{2}})$ has a nonsplit link of two 0-homologous cycles.\\

Conversely, assume $f(G \setminus V_{0})$ and $f(G  \setminus V_{1})$ are equivalent to nonseparating planar embeddings. Suppose $f(G+\bar{K_{2}})$ has a nonsplit link of two 0-homologous cycles, then $f(G)$ is a separating projective planar embedding and has a separating cycle $C$. Note that $C$ either has no 1-homologous edge or has two 1-homologous edges. If $C$ has no 1-homologous edge, then $f(G \setminus V_{1})$ is a separating planar embedding. If $C$ has two 1-homologous edges, then $f(G\setminus V_{0})$ can be isotoped to a separating affine embedding. In both cases we have a contradiction.\\
\end{pro}

\subsubsection{Case 2}

Assume $f(G)$ has the embedding as in Case 2, then $f(G+\bar{K_{2}})$ has no link between two 1-homologous cycles. It suffices to look at links between two 0-homologous cycles and between one 0-homologous and one 1-homologous cycle.\\

\begin{theo}
Assume $f(G)$ has the embedding as in Case 2. Let the set of 1-homologous edges be $(v_{0}, v_{1})$, $(v_{1}, v_{2})$, and $(v_{2}, v_{0})$. The embedding $f(G+\bar{K_{2}})$ has no nonsplit link of two 0-homologous cycles if and only if $f(G \setminus (v_{0}, v_{1}))$, $f(G \setminus (v_{1}, v_{2}))$ and $f(G \setminus (v_{2}, v_{0}))$ are equivalent to nonseparating planar embeddings.
\end{theo}
\begin{pro}
 Suppose $f(G \setminus (v_{0}, v_{1}))$, $f(G \setminus (v_{1}, v_{2}))$ and $f(G \setminus (v_{2}, v_{0}))$ has a separating cycle, then $f(G)$ has a separating cycle. By Theorem \ref{Theorem1}, $f(G+\bar{K_{2}})$ has a nonsplit link of two 0-homologous cycles.\\

Suppose $f(G+\bar{K_{2}})$ has a nonsplit link of two 0-homologous cycles, then $f(G)$ is a separating projective planar embedding and has a separating cycle $C$. Note that $C$ either has no 1-homologous edge or has two 1-homologous edges. If $C$ has no 1-homologous edge, then $f(G \setminus \{(v_{0}, v_{1}), (v_{1}, v_{2}), (v_{2}, v_{0})\})$ is a separating planar embedding. If $C$ has two 1-homologous edges, then one of $f(G \setminus (v_{0}, v_{1}))$, $f(G \setminus (v_{1}, v_{2}))$ and $f(G \setminus (v_{2}, v_{0}))$ is a separating planar embedding. In both cases we have a contradiction.\\
\end{pro}
\begin{note}
Dehkordi and Farr \cite{Dehkordi} showed that every nonseparating planar graph is an outerplanar graph, a wheel graph, or an elongated triangular prism graph. We can then separately look at cases where $f(G \setminus V_{0})$, $f(G \setminus V_{1})$, $f(G \setminus (v_{0}, v_{1}))$, $f(G \setminus (v_{1}, v_{2}))$ and $f(G \setminus (v_{2}, v_{0}))$ are outerplanar, wheel, or elongated triangular prism. 
\end{note}





\subsection{IPL graphs on 16 edges}

There is one minor-minimal IPL graph with 15 edges, which is $K_{4,4}-e$ \cite{2008REU}. We examined the other Petersen family graphs with one added edge or one vertex splitting to search for a minor-minimal IPL graph on 16 edges. The details of this examination follow. We found no new minor-minimal IPL graphs on 16 edges. We declined to examine possible minor-minimal IPL graphs on 17 or more edges.\\

Since $K_{6}$ is not IPL and $K_7-2e$ is minor minimal IPL, there is no other IPL graphs with seven or fewer vertices. However, classifying all minor-minimal IPL graphs on 8 (or more) vertices remains open.\\

Here we show every graph that results from a Petersen family graph with either one added edge or one vertex splitting. This makes the graphs have 16 edges. Every graph here is either not a minor-minimal IPL graph, or contains $K_{4,4}-e$ as a minor, and this shows there is no minor-minimal IPL graph on 16 edges.

\begin{figure}[H]
\centering
\includegraphics[scale=0.4]{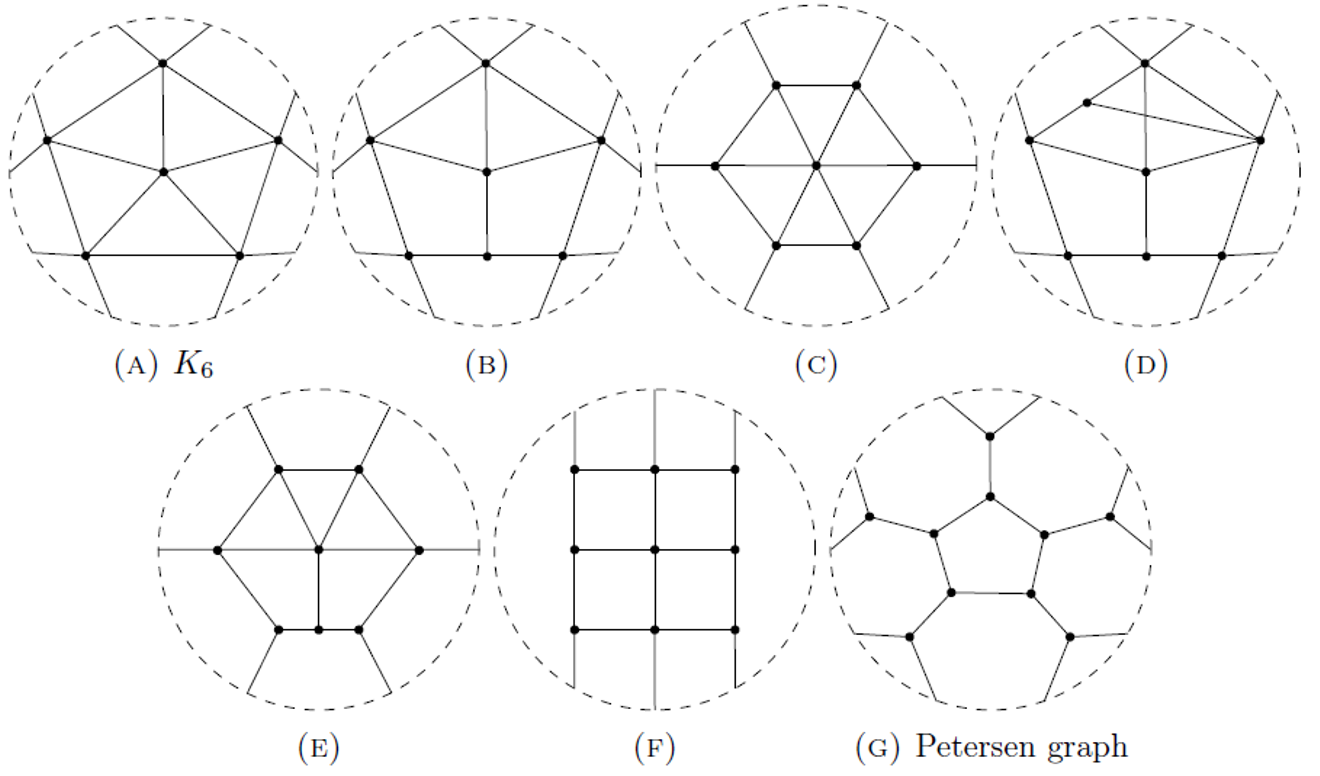}
\caption{The Petersen family graphs on the projective plane. Image from \cite{Dehkordi}, which was inspired by K. Kawarabayashi et al \cite{Kawarabayashi}.}
\end{figure}

\subsubsection{Adding an edge}

We will add one edge to each of the Petersen family graphs and determine if the resulting graph is minor-minimally IPL.

First, consider graph \textbf{$K_{6}$(A)}.
No edge can be added to this graph because it is a complete graph.

Now, consider graph \textbf{$P_{7}$(B) and $K_{3,3,1}$(C)}.
The graph $K_7 -2e$ is also minor-minimally IPL \cite{2008REU}. This graph has seven vertices and 19 edges. This means that there is no IPL graph with seven vertices and less than 19 edges. Thus, we do not need to examine the graphs $P_7$ and $K_{3,3,1}$ for IPL graphs on 16 edges, as both have seven vertices.

Now, consider graph \textbf{$P_{8}$(E)}.
The graph $P_8$ contains three vertices of degree 3, four vertices of degree 4, and one vertex of degree 5. One of the vertices of degree 3, vertex 6, is not equivalent to the others under any graph automorphism as in Figure \ref{6connected}. This vertex is already connected to vertices 5, 7, and 8. If we connect vertex 6 to the remaining vertices, the graph remains projective planar. The remaining two vertices of degree 3 are equivalent, so without loss of generality, we have chosen vertex 7. This vertex is already connected to vertices 1, 3, and 6. If we connect vertex 6 to the remaining vertices, the graph remains projective planar, as seen in Figure \ref{7connected}.

\begin{figure}[H]
\centering 
\begin{minipage}[b]{0.4\linewidth}
\centering
\includegraphics[scale=0.66]{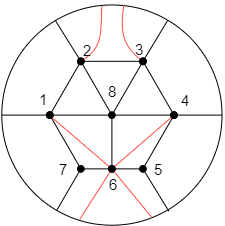}
\caption{Vertex 6 connected}
\label{6connected}
\end{minipage}
\quad
\begin{minipage}[b]{0.4\linewidth}
\centering
\includegraphics[scale=0.33]{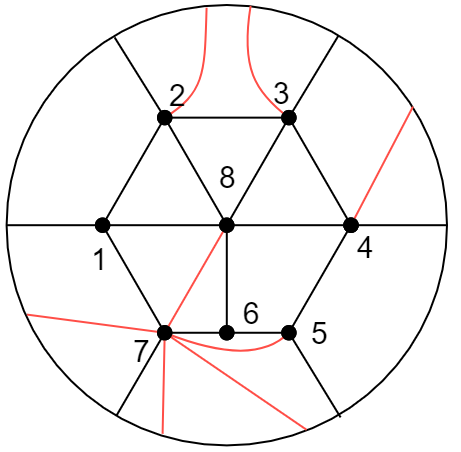} 
\caption{Vertex 7 connected}
\label{7connected} 
\end{minipage}
\end{figure}

The vertices of degree 4 are equivalent, so without loss of generality, we have chosen vertex 1. It is already connected to vertices 2, 4, 7, and 8. If we connect vertex 1 to the remaining vertices, the graph remains projective planar, as seen in Figure \ref{1connected}. There is one vertex of degree 5, which is vertex 8. It is already connected to 1, 2, 3, 4, and 6. If we connect it to 5 and 7, the graph remains projective planar, as seen in Figure \ref{8connected}. Thus, adding one edge to $P_{8}$ will not create a minor-minimally IPL graph.

\begin{figure}[H]
\centering
\begin{minipage}[b]{0.4\linewidth}
\centering
\includegraphics[scale=0.35]{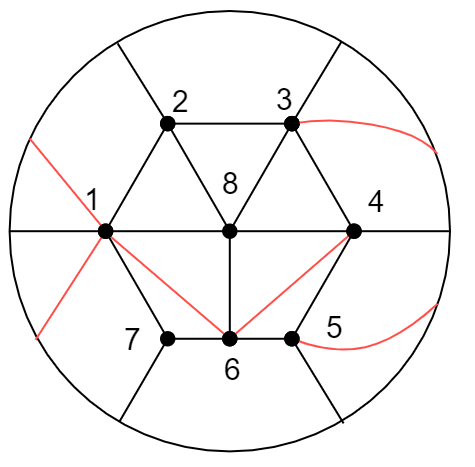}
\caption{Vertex 1 connected}
\label{1connected}
\end{minipage}
\quad
\begin{minipage}[b]{0.4\linewidth}
\centering
\includegraphics[scale=0.35]{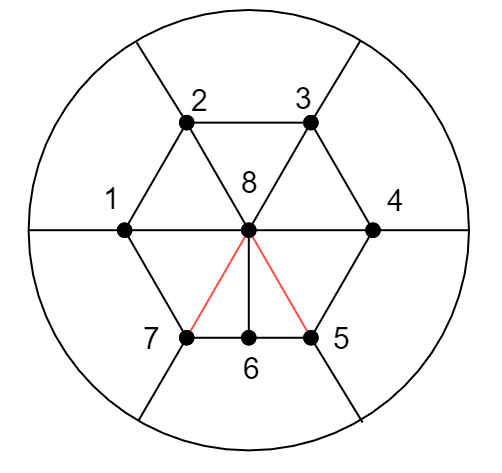}
\caption{Vertex 8 connected}
\label{8connected}
\end{minipage}
\end{figure}

Now, consider graph \textbf{$P_{9}$(F)}.
In $P_9$, there are six vertices of degree 3 and three vertices of degree 4. The vertices of degree 3 are all equivalent, so without loss of generality, we have chosen vertex 1. This vertex is already connected to vertices 2, 4, and 9. If we add an edge from vertex 1 to any of the remaining vertices, the graph will remain projective planar, as seen in Figure \ref{v1connected}. Now consider the vertices of degree 4. They are all equivalent, so without loss of generality, we have chosen vertex 4. This vertex is already connected to vertices 1, 5, 6, and 7. As seen in the drawing, if we add an edge from vertex 4 to any of the remaining vertices, the graph will remain projective planar, as seen in Figure \ref{v4connected}. Thus, adding one edge to $P_{9}$ will not create a minor-minimally IPL graph.
\begin{figure}[H]
\centering
\begin{minipage}[b]{0.3\linewidth}
\centering
\includegraphics[scale=0.35]{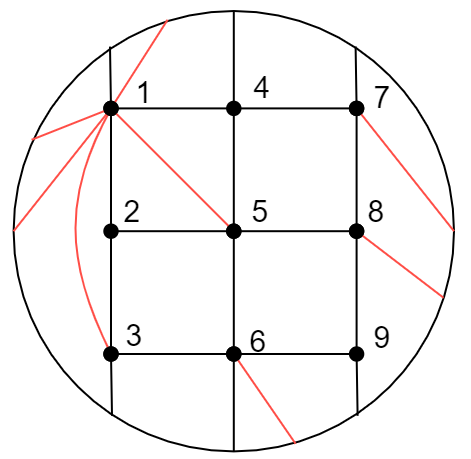}
\caption{Vertex 1 connected}
\label{v1connected}
\end{minipage}
\quad
\begin{minipage}[b]{0.3\linewidth}
\centering
\includegraphics[scale=0.35]{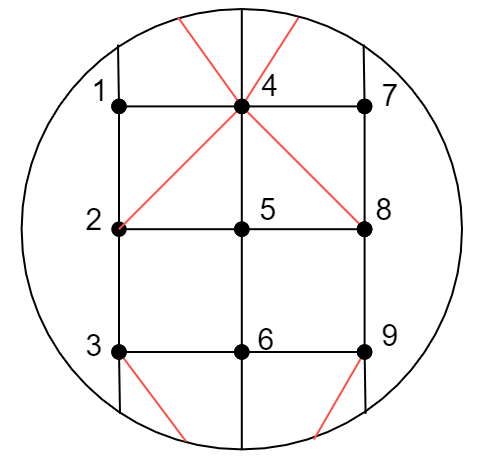}
\caption{Vertex 4 connected}
\label{v4connected}
\end{minipage}
\quad
\begin{minipage}[b]{0.3\linewidth}
\centering
\includegraphics[scale=0.35]{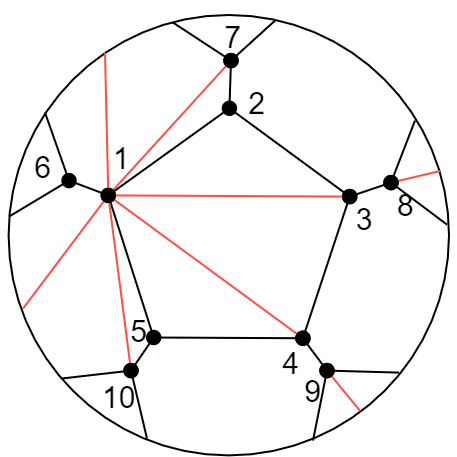}
\caption{Vertex 1 connected}
\label{p10v1connected}
\end{minipage}
\end{figure}

Finally, consider graph \textbf{$P_{10}$(G)}.
Every vertex on this graph has degree 3, and they are all equivalent. Without loss of generality, we have chosen vertex 1. This vertex is already connected to vertices 2, 5, and 6. If we add an edge from vertex 1 to any of the remaining vertices, the graph will remain projective planar, as seen in Figure \ref{p10v1connected}. Thus, adding one edge to $P_{10}$ will not create a minor-minimally IPL graph.

\subsubsection{Splitting a vertex}
Vertex splitting is the opposite of edge contraction. When splitting a vertex, the edges from that vertex are divided into two groups - the edges that will remain connected to the old vertex and the edges that will be moved and now connect to the new vertex. Splitting a vertex of degree 3 means that one of these groups will have only one edge in it. This results in a graph that is homeomorphic to the original graph. The original graphs are all projective planar, and the homeomorphic graphs will also have this property. Thus, for the vertices of degree 3 we know the graphs that result from their vertex splitting will not be IPL.

First, consider graph \textbf{$K_{6}$(A)}.
Every vertex splitting of $K_6$ will create a graph on seven vertices. The only minor-minimal IPL graph on seven vertices are $K_7 - 2e$, which have 19 edges.

Now, consider graph \textbf{$P_{7}$(B)}.
This graph has one vertex of degree 3, three vertices of degree 4, and three vertices of degree 5. We consider the vertices of degree 4. They are all equivalent. Without loss of generality, we have chosen vertex 7. Assign numbers to its adjacent edges: edge 1 connects vertex 7 to vertex 1, edge 2 connects vertex 7 to vertex 2, edge 3 connects vertex 7 to vertex 3, and edge 4 connects vertex 7 to vertex 5, as in Figure \ref{p7labelled}.

\begin{figure}[H]
\centering
\includegraphics[scale=0.35]{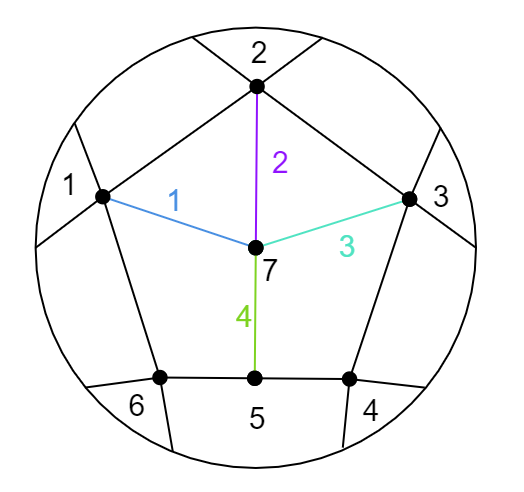}
\caption{Graph $P_7$ with the edges incident to vertex 7 labelled.}
\label{p7labelled}

\end{figure}

Cases 1: Suppose three of the four edges connect to the new vertex. This will create a homeomorphic graph, so this case will create a projective planar graph.

Case 2: Suppose two edges that bound the same region connect to the new vertex. This case will create a projective planar graph, as seen in Figure \ref{p7case2}.

Case 3: Suppose two edges that do not bound the same regions, one of which connects to a degree 3 vertex, are connected to the new vertex. This first embedding is not projective planar, as seen in Figure \ref{Case 3 Embedding 1_1}. However, there is another embedding that is projective planar, as seen in Figure \ref{Case 3 Embedding 2_1}, so this case does not create an IPL graph.
\begin{figure}[H]
\centering
\begin{minipage}[b]{0.3\linewidth}
\centering
\includegraphics[scale=0.35]{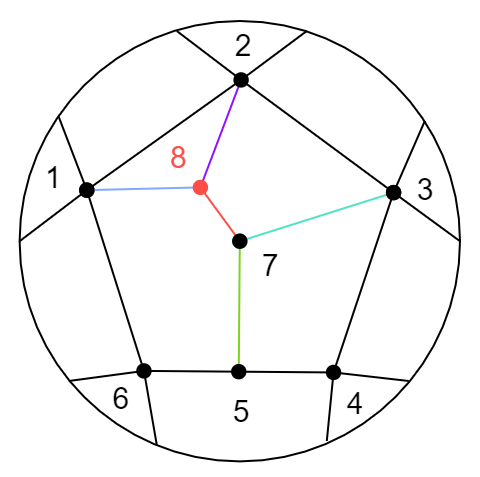}
\caption{Case 2 embedding}
\label{p7case2}
\end{minipage}
\quad
\begin{minipage}[b]{0.32\linewidth}
\centering
\includegraphics[scale=0.35]{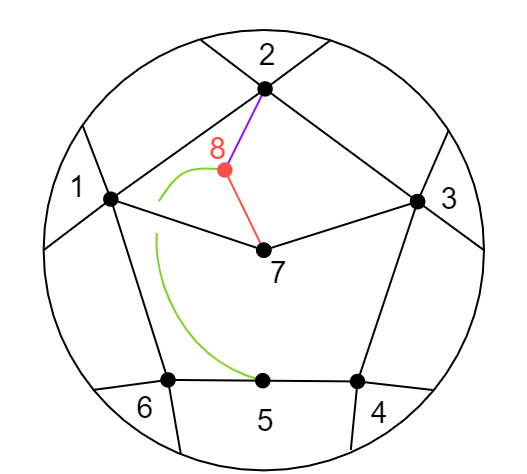}
\caption{Case 3 Embedding 1}
\label{Case 3 Embedding 1_1}
\end{minipage}
\quad
\begin{minipage}[b]{0.32\linewidth}
\centering
\includegraphics[scale=0.35]{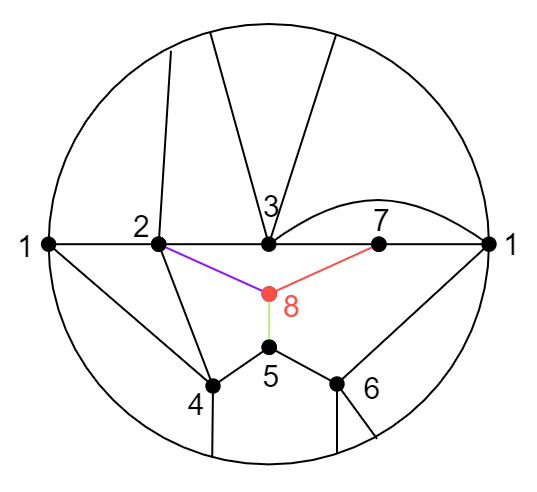}
\caption{Case 3 Embedding 2}
\label{Case 3 Embedding 2_1}
\end{minipage}
\end{figure}

Case 4: Suppose two edges that do not bound the same regions, both of which connect to a degree 5 vertex, are connected to the new vertex. This case is equivalent to Case 3, except the new vertex and vertex 7 are switched - the graphs are isomorphic. Therefore, this graph will also be projective planar.

From these four cases, we can conclude that no splitting of a degree 4 vertex will create a minor-minimal IPL graph with 16 edges. Now we must consider the vertices of degree 5. They are all equivalent, so without loss of generality, we have chosen vertex 2. Assign numbers to its adjacent edges: Edge 1 connects vertex 2 to vertex 1, Edge 2 connects vertex 2 to vertex 4, Edge 3 connects vertex 2 to vertex 6, Edge 4 connects vertex 2 to vertex 3, and Edge 5 connects vertex 2 to vertex 7.

\begin{figure}[H]
\centering
\includegraphics[scale=0.35]{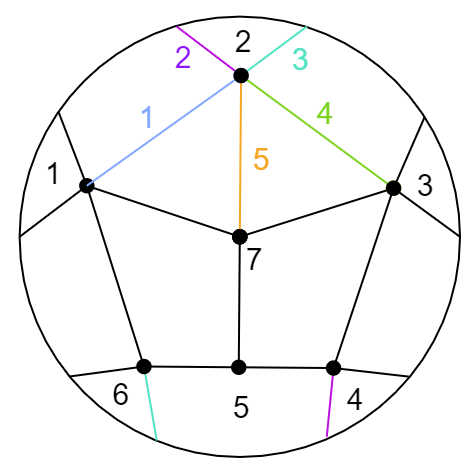}
\caption{Graph $P_7$ with the edges incident to vertex 2 labelled.}

\end{figure}

Case 1: Suppose four of the five edges connect to the new vertex. This will create a homeomorphic graph, so this case will create a projective planar graph.

Case 2: Suppose two edges that bound the same region connect to the new vertex. This case will create a projective planar graph, as seen in Figure \ref{Case 2 Embedding}.

Case 3: Suppose two edges that bound the same region remain connected to the old vertex, and the remaining three edges connect to the new vertex. This case is equivalent to Case 2, except the new vertex and vertex 2 are switched - the graphs are isomorphic. This means that this case will also create a projective planar graph.

Case 4: Suppose the two edges that connect to the degree 5 vertices are connect to the new vertex. This means edges 1 and 4 connect to the new vertex. This graph is nonprojective planar, as seen in Figure \ref{case4}, but it is not minor-minimally IPL, as this graph has $K_{4,4}-e$ as a subgraph. The minor $K_{4,4}-e$ is formed by deleting the edge that connects 1 and 3.
\begin{figure}[H]
\centering
\begin{minipage}[b]{0.4\linewidth}
\centering
\includegraphics[scale=0.35]{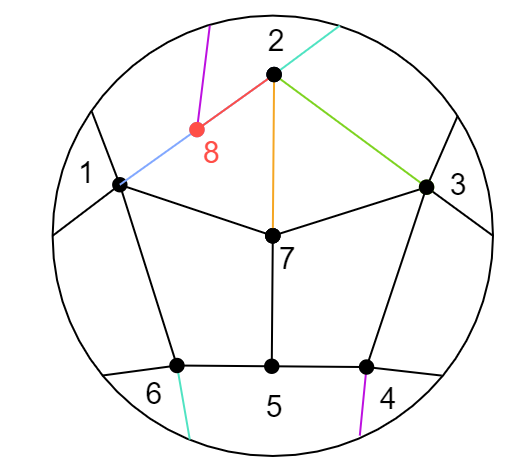}
\caption{Case 2 Embedding}
\label{Case 2 Embedding}
\end{minipage}
\quad
\begin{minipage}[b]{0.4\linewidth}
\centering
\includegraphics[scale=0.35]{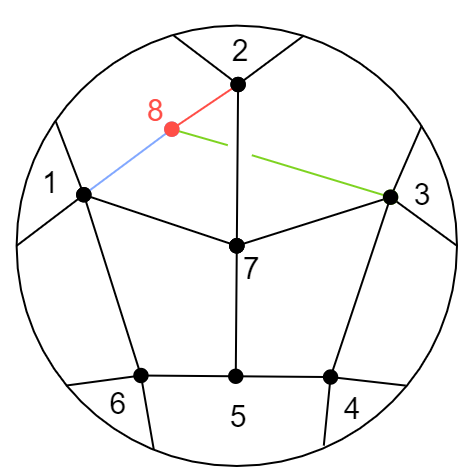}
\caption{Case 4 Embedding}
\label{case4}
\end{minipage}
\end{figure}

The graph in Figure \ref{case4} is equivalent to Graph (D) from the image in the introduction. The graphs are isomorphic, as they are the same graph when the vertices are labelled like this:

\begin{figure}[H]
\centering
\includegraphics[scale=0.35]{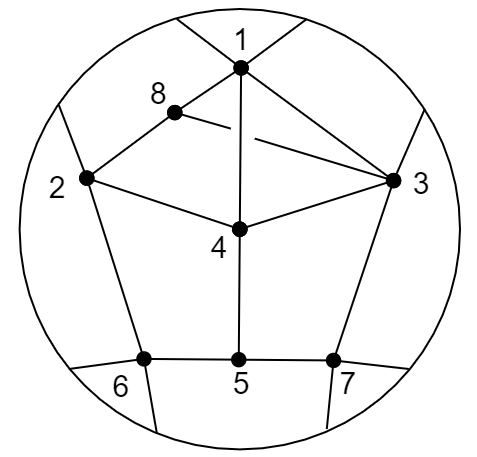}
    \caption{Equivalent embedding to Graph (D)}
\end{figure}

Case 5: Suppose the three edges that connect to degree 4 vertices are connected to the new vertex. This case is equivalent to Case 4, except the new vertex and vertex 2 are switched - the graphs are isomorphic. This means this graph is also IPL, but not minor-minimally.

Case 6: Suppose two edges that do not bound the same region, one of which connects to a degree 4 vertex and the other connects to a degree 5 vertex, are connected to the new vertex. Without loss of generality, we have chosen edges 1 and 3. This first embedding is not projective planar, as seen in Figure \ref{case6_1}. However, there is another embedding that is projective planar, as seen in Figure \ref{case6_2}, so this case does not create an IPL graph.
\begin{figure}[H]
\centering
\begin{minipage}[b]{0.45\linewidth}
\centering
\includegraphics[scale=0.35]{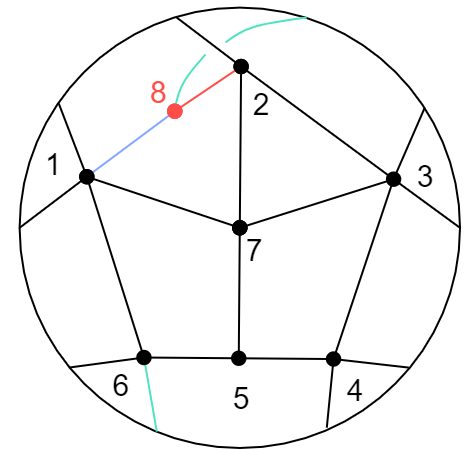}
\caption{Case 6 Embedding 1}
\label{case6_1}
\end{minipage}
\quad
\begin{minipage}[b]{0.45\linewidth}
\centering
\includegraphics[scale=0.35]{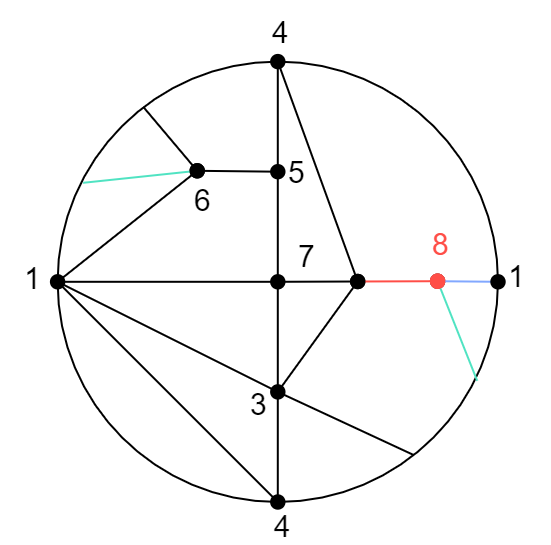}
\caption{Case 6 Embedding 2}
\label{case6_2}
\end{minipage}
\end{figure}

Case 7: Suppose two edges that do not bound the same region, one of which connects to a degree 4 vertex and the other connects to a degree 5 vertex, remain connected to vertex 2 and the remaining three edges are connected to the new vertex. This case is equivalent to Case 6, except the new vertex and vertex 2 are switched - the graphs are isomorphic. This means that this case will also create a projective planar graph.

Case 8: Suppose two edges that connect to the degree 5 vertices and do not bound the same region are connected to the new vertex. Without loss of generality, we have chosen edges 2 and 5 to connect to the new vertex. This first embedding is not projective planar, as seen in Figure \ref{case81}. However, there is another embedding that is projective planar, as seen in Figure \ref{case82}, so this case does not create an IPL graph.
\begin{figure}[H]
\centering
\begin{minipage}[b]{0.45\linewidth}
\centering
\includegraphics[scale=0.35]{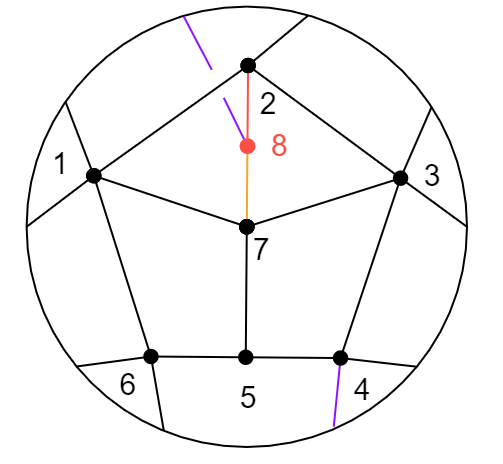}
\caption{Case 8 Embedding 1}
\label{case81}
\end{minipage}
\quad
\begin{minipage}[b]{0.45\linewidth}
\centering
\includegraphics[scale=0.35]{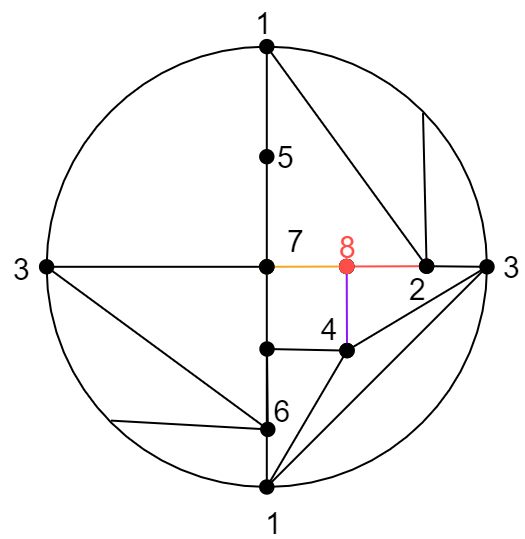}
\caption{Case 8 Embedding 2}
\label{case82}
\end{minipage}
\end{figure}

From these cases, we can conclude that no splitting of a degree 5 vertex will create a minor-minimal IPL graph with 16 edges. Thus, we can conclude none of graphs that contain $P_7$ as a minor with 16 edges that result from vertex splittings are minor-minimally IPL.

Now, consider graph \textbf{$K_{3,3,1}$(C)}.
$K_{3,3,1}$ has six vertices of degree 4 and one vertex of degree 6. The six vertices of degree 4 are all equivalent, so without loss of generality, we have chosen vertex 1. Assign numbers to its adjacent edges: edge 1 connects vertex 1 to vertex 2, edge 2 connects vertex 1 to vertex 7, edge 3 connects vertex 1 to vertex 6, and edge 4 connects vertex 1 to vertex 4.
\begin{figure}[H]
\centering
\includegraphics[scale=0.38]{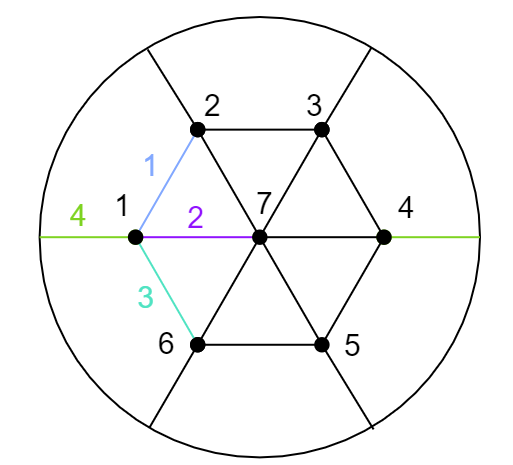}
\caption{Graph $K_{3,3,1}$ with the edges incident to vertex 1 labelled.}
\end{figure}

Case 1: Suppose three of the four edges connect to the new vertex. This will create a homeomorphic graph, so this case will create a projective planar graph.

Case 2: Suppose two edges that bound the same region connect to the new vertex. This case will create a projective planar graph, as seen in Figure \ref{k331case2}.

Case 3: Suppose two edges that do not bound the same regions, one of which connects to a degree 6 vertex, are connected to the new vertex. This first embedding is not projective planar, as seen in Figure \ref{case3emb1}. However, there is another embedding that is projective planar, as seen in Figure \ref{case3emb2}, so this case does not create an IPL graph.
\begin{figure}[H]
\centering
\begin{minipage}[b]{0.3\linewidth}
\centering
\includegraphics[scale=0.35]{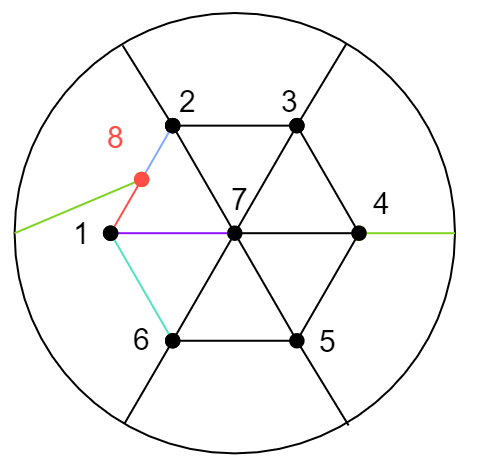}
\caption{Case 2 Embedding}
\label{k331case2}
\end{minipage}
\quad
\begin{minipage}[b]{0.32\linewidth}
\centering
\includegraphics[scale=0.35]{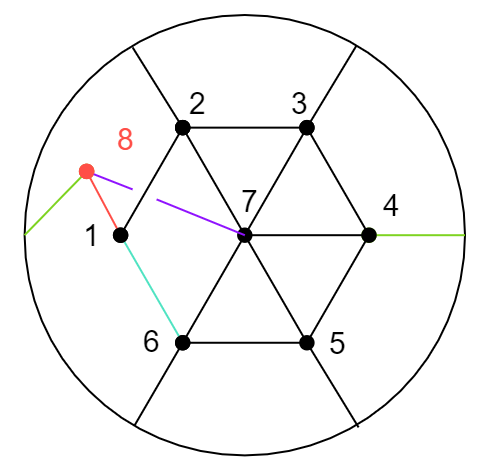}
\caption{Case 3 Embedding 1}
\label{case3emb1}
\end{minipage}
\quad
\begin{minipage}[b]{0.32\linewidth}
\centering
\includegraphics[scale=0.6]{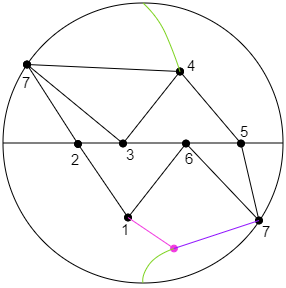}
\caption{Case 3 Embedding 2}
\label{case3emb2}
\end{minipage}
\end{figure}

Case 4: Suppose two edges that do not bound the same regions, both of which connect to degree 4 vertices, are connected to the new vertex. This case is equivalent to Case 3, except the new vertex and vertex 1 are switched - the graphs are isomorphic. Therefore, this graph will also create a projective planar graph.

We can conclude that no splitting of a degree 4 vertex of $K_{3,3,1}$ will create a minor-minimal IPL graph with 16 edges.

Next, we will examine the degree 6 vertex, vertex 7. Assign numbers to its adjacent edges: edge 1 connects vertex 7 to vertex 1, edge 2 connects vertex 7 to vertex 2, edge 3 connects vertex 7 to vertex 3, edge 4 connects vertex 7 to vertex 4, edge 5 connects vertex 7 to vertex 5, and edge 6 connects vertex 7 to vertex 6.
\begin{figure}[H]
\centering
\includegraphics[scale=0.62]{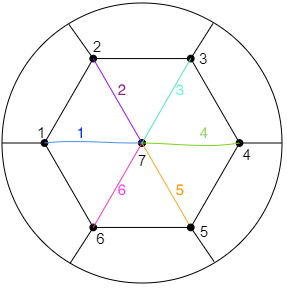}
\caption{Graph $K_{3,3,1}$ with the edges incident to vertex 7 labelled.}
\end{figure}

Case 1: Suppose five of the six edges connect to the new vertex. This will create a homeomorphic graph, so this case will create a projective planar graph.

Case 2: Suppose two edges that bound the same region connect to the new vertex. This case will create a projective planar graph, as seen in Figure \ref{case2embedd}.

Case 3: Suppose two edges that bound the same region remain connected to the old vertex, and the remaining four edges connect to the new vertex. This case is equivalent to Case 2, except the new vertex and vertex 7 are switched - the graphs are isomorphic. This means that this case will also create a projective planar graph.

Case 4: Suppose three edges that bound adjacent regions connect to the new vertex. This case will create a projective planar graph, as seen in Figure \ref{case4embedd}.

\begin{figure}[H]
\centering
\begin{minipage}[b]{0.3\linewidth}
\centering
\includegraphics[scale=0.62]{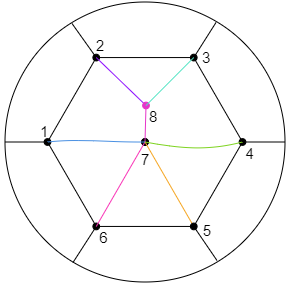}
\caption{Case 2 Embedding}
\label{case2embedd}
\end{minipage}
\quad
\begin{minipage}[b]{0.32\linewidth}
\centering
\includegraphics[scale=0.62]{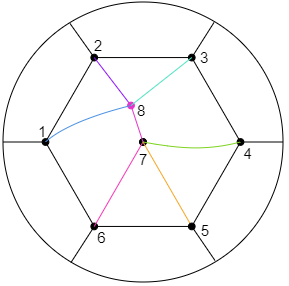}
\caption{Case 4 Embedding}
\label{case4embedd}
\end{minipage}
\quad
\begin{minipage}[b]{0.32\linewidth}
\centering
\includegraphics[scale=0.62]{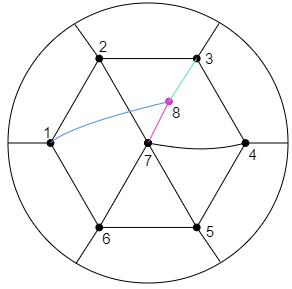}
\caption{Case 5 Embedding}
\label{Case5embed}
\end{minipage}
\end{figure}

Case 5: Suppose two edges that are not adjacent, separated by one edge between their regions, connect to the new vertex. Without loss of generality, chose edges 1 and 3. This embedding is not projective planar, as seen in Figure \ref{Case5embed}, but it is not minor-minimally IPL, because it has $K_{4,4}-e$ as a minor. Delete the edge that connects vertices 5 and 7 to get $K_{4,4}-e$.

Case 6: Suppose two edges that are not adjacent, separated by one edge between their regions, remain connected to vertex 7 and the remaining edges connect to the new vertex. Without loss of generality, chose edges 2, 4, 5, and 6 to connect to the new vertex. This case is equivalent to Case 5, except the new vertex and vertex 7 are switched - the graphs are isomorphic. This means that the graph is IPL, but not minor-minimally.

Case 7: Suppose two edges that are not adjacent, separated by two edges between their regions, connect to the new vertex. Without loss of generality, we have chosen edges 1 and 4. This first embedding is not projective planar, as seen in Figure \ref{Case7em1}. However, there is another embedding that is projective planar, as seen in Figure \ref{Case7em2}, so this case does not create an IPL graph.
\begin{figure}[H]
\centering
\begin{minipage}[b]{0.45\linewidth}
\centering
\includegraphics[scale=0.62]{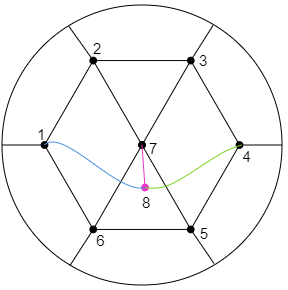}
\caption{Case 7 Embedding 1}
\label{Case7em1}
\end{minipage}
\quad
\begin{minipage}[b]{0.45\linewidth}
\centering
\includegraphics[scale=0.62]{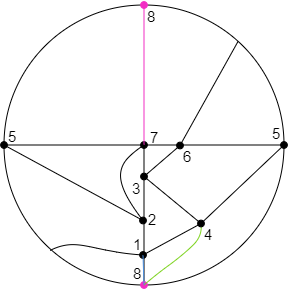}
\caption{Case 7 Embedding 2}
\label{Case7em2}
\end{minipage}
\end{figure}

Case 8: Suppose two edges that are not adjacent, separated by two edges between their regions, remain connected to vertex 7 and the remaining edges connect to the new vertex. Without loss of generality, we have chosen edges 2, 3, 5, and 6 to connect to the new vertex. This case is equivalent to Case 7, except the new vertex and vertex 7 are switched - the graphs are isomorphic. Therefore, this case will also create a projective planar graph.

Case 9: Suppose two edges that bound the same region and a third edge that is the opposite of either of those two edges connect to the new vertex. Without loss of generality, we have chosen edges 1, 2, and 4 to connect to the new vertex. This first embedding is not projective planar, as seen in Figure \ref{Case91}. However, there is another embedding that is projective planar, as seen in Figure \ref{Case92}, so this case does not create an IPL graph.
\begin{figure}[H]
\centering
\begin{minipage}[b]{0.45\linewidth}
\centering
\includegraphics[scale=0.62]{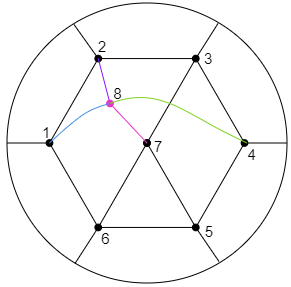}
\caption{Case 9 Embedding 1}
\label{Case91}
\end{minipage}
\quad
\begin{minipage}[b]{0.45\linewidth}
\centering
\includegraphics[scale=0.62]{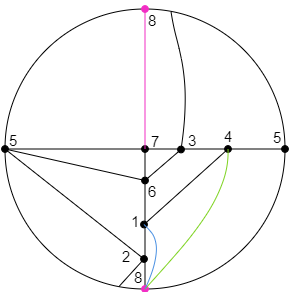}
\caption{Case 9 Embedding 2}
\label{Case92}
\end{minipage}
\end{figure}

Case 10: Suppose three edges, where no two edges bound the same region, connect to the new vertex. Without loss of generality, we have chosen edges 1, 3, and 5 connect to the new vertex. This embedding is not projective planar, as seen in Figure \ref{Case10}. It is the graph $K_{4,4}$. Thus, it is not a new minor-minimally IPL, as $K_{4,4}-e$ is a minor.
\begin{figure}[H]
\centering
\includegraphics[scale=0.62]{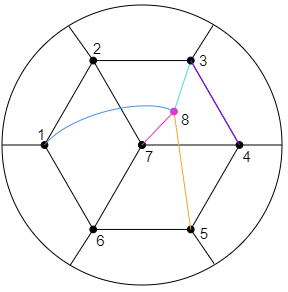}
\caption{Case 10, which is $K_{4,4}$}
\label{Case10}
\end{figure}

We can conclude that no splitting of the degree 7 vertex of $K_{3,3,1}$ will create a minor-minimal IPL graph with 16 edges. Thus, we can conclude none of the graphs containing $K_{3,3,1}$ as a minor on 16 edges that result from vertex splittings are minor-minimally IPL.

Now, consider graph \textbf{$P_{8}$(E)}. This graph has three vertices of degree 3, four vertices of degree 4, and one vertex of degree 5. First, we consider the degree 4 vertices. They are all equivalent, so without loss of generality, we have chosen vertex 3. Assign numbers to its adjacent edges: edge 1 connects vertex 3 to vertex 7, edge 2 connects vertex 3 to vertex 4, edge 3 connects vertex 3 to vertex 8, and edge 4 connects vertex 3 to vertex 2.
\begin{figure}[H]
\centering
\includegraphics[scale=0.62]{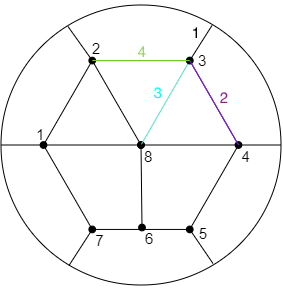}
\caption{Graph $P_8$ with the edges incident to vertex 3 labelled.}
\end{figure}

Case 1: Suppose three of the four edges connect to the new vertex. This will create a homeomorphic graph, so this case will create a projective planar graph.

Case 2: Suppose two edges that bound the same region connect to the new vertex. This case will create a projective planar graph, as seen in Figure \ref{p8case2}.

Case 3: Suppose opposite vertices, each from a degree 4 vertex, connect to the new vertex. For this example, edges 2 and 4 connect to the new vertex. This first embedding is not projective planar, as seen in Figure \ref{Case3p8_emb1}. However, there is another embedding that is projective planar, as seen in Figure \ref{Case3p8_emb2}, so this case does not create an IPL graph.
\begin{figure}[H]
\centering
\begin{minipage}[b]{0.3\linewidth}
\centering
\includegraphics[scale=0.62]{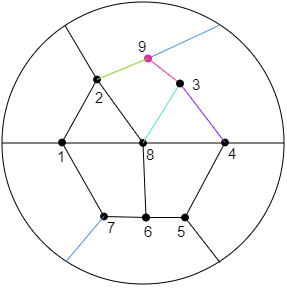}
\caption{Case 2 Embedding}
\label{p8case2}
\end{minipage}
\quad
\begin{minipage}[b]{0.32\linewidth}
\centering
\includegraphics[scale=0.62]{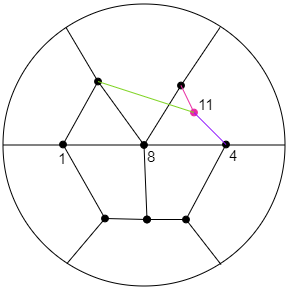}
\caption{Case 3 Embedding 1}
\label{Case3p8_emb1}
\end{minipage}
\quad
\begin{minipage}[b]{0.32\linewidth}
\centering
\includegraphics[scale=0.62]{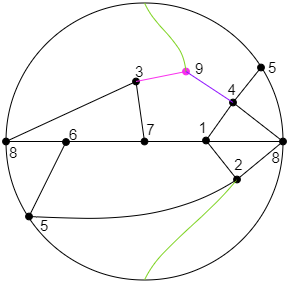}
\caption{Case 3 Embedding 2}
\label{Case3p8_emb2}
\end{minipage}
\end{figure}

Case 4: Suppose opposite vertices, one from a degree 3 vertex and one from a degree 5 vertex, connect to the new vertex. For this example, this would mean edges 1 and 3 connect to the new vertex. This case is equivalent to Case 3, except the new vertex and vertex 3 are switched - the graphs are isomorphic. Therefore, this graph will also be projective planar.

We can conclude that no splitting of the degree 4 vertices of $P_{8}$ will create a minor-minimal IPL graph with 16 edges. Next we will look at the vertex of degree 5, vertex 8. Assign numbers to its adjacent edges: Edge 1 connects vertex 8 to vertex 4, Edge 2 connects vertex 8 to vertex 6, Edge 3 connects vertex 8 to vertex 1, Edge 4 connects vertex 8 to vertex 2, and Edge 5 connects vertex 8 to vertex 3.
\begin{figure}[H]
\centering
\includegraphics[scale=0.62]{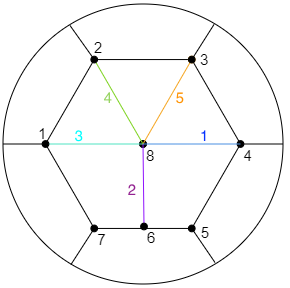}
\caption{Graph $P_8$ with the edges incident to vertex 8 labelled.}
\end{figure}

Case 1: Suppose four of the five edges connect to the new vertex. This will create a homeomorphic graph, so this case will create a projective planar graph.

Case 2: Suppose two edges that bound the same region connect to the new vertex. This case will create a projective planar graph, as seen in Figure \ref{p8case22}.

Case 3: Suppose two edges that bound the same region remain connected to the old vertex, and the remaining three edges connect to the new vertex. This case is equivalent to Case 2, except the new vertex and vertex 8 are switched - the graphs are isomorphic. This means that this case will also create a projective planar graph.

Case 4: Suppose the two opposite edges that connect to degree 4 vertices connect to the new vertex. This means edges 1 and 3 connect to the new vertex. This first embedding is not projective planar, as seen in Figure \ref{case4p8emb1}. However, there is another embedding that is projective planar, as seen in Figure \ref{case4p8emb2}, so this case does not create an IPL graph.
\begin{figure}[H]
\centering
\begin{minipage}[b]{0.3\linewidth}
\centering
\includegraphics[scale=0.62]{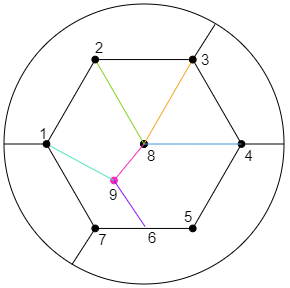}
\caption{Case 2 Embedding}
\label{p8case22}
\end{minipage}
\quad
\begin{minipage}[b]{0.32\linewidth}
\centering
\includegraphics[scale=0.62]{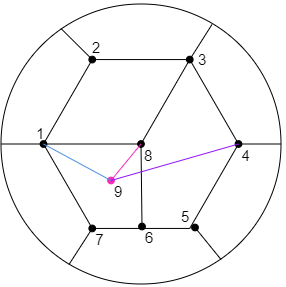}
\caption{Case 4 Embedding 1}
\label{case4p8emb1}
\end{minipage}
\quad
\begin{minipage}[b]{0.32\linewidth}
\centering
\includegraphics[scale=0.62]{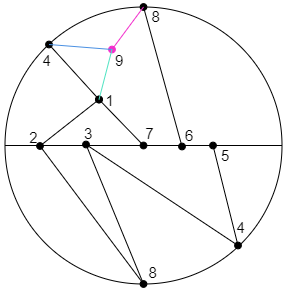}
\caption{Case 4 Embedding 2}
\label{case4p8emb2}
\end{minipage}
\end{figure}

Case 5: Suppose the two opposite edges that connect to degree 4 vertices remain connected to vertex 8, and the remaining edges connect to the new vertex. This means edges 2, 4, and 5 connect to the new vertex. This case is equivalent to Case 4, except the new vertex and vertex 8 are switched - the graphs are isomorphic. This means this case also creates a projective planar graph.

Case 6: Suppose two edges that are not opposites and that connect to degree 4 vertices connect to the new vertex. Without loss of generality, we have chosen edges 1 and 4. This first embedding is not projective planar, as seen in Figure \ref{case6p8emb1}. However, there is another embedding that is projective planar, as seen in Figure \ref{case6p8emb2}, so this case does not create an IPL graph.
\begin{figure}[H]
\centering
\begin{minipage}[b]{0.45\linewidth}
\centering
\includegraphics[scale=0.62]{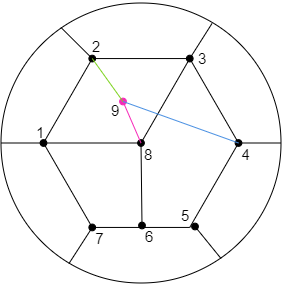}
\caption{Case 6 Embedding 1}
\label{case6p8emb1}
\end{minipage}
\quad
\begin{minipage}[b]{0.45\linewidth}
\centering
\includegraphics[scale=0.62]{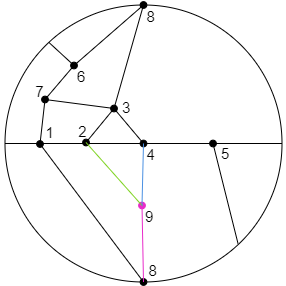}
\caption{Case 6 Embedding 2}
\label{case6p8emb2}
\end{minipage}
\end{figure}

Case 7: Suppose two edges that are not opposites and that connect to degree 4 vertices remain connected to vertex 8 and the remaining 3 edges are connected to the new vertex. This case is equivalent to Case 6, except the new vertex and vertex 8 are switched - the graphs are isomorphic. This means that this case will also create a projective planar graph.

Case 8: Suppose two edges that do not bound the same region, one from a degree 3 vertex and one from a degree 4 vertex, are connected to the new vertex. Without loss of generality, we have chosen edges 2 and 4 to connect to the new vertex. This first embedding is not projective planar, as seen in Figure \ref{case8emb1}. However, there is another embedding that is projective planar, as seen in Figure \ref{case8emb2}, so this case does not create an IPL graph.
\begin{figure}[H]
\centering
\begin{minipage}[b]{0.45\linewidth}
\centering
\includegraphics[scale=0.62]{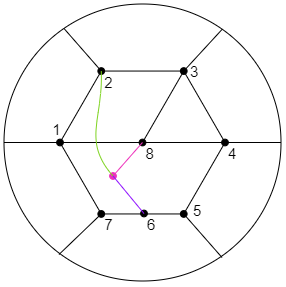}
\caption{Case 8 Embedding 1}
\label{case8emb1}
\end{minipage}
\quad
\begin{minipage}[b]{0.45\linewidth}
\centering
\includegraphics[scale=0.62]{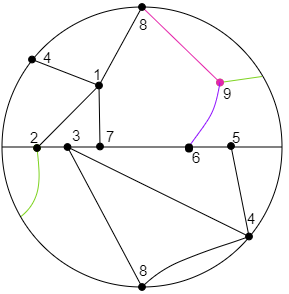}
\caption{Case 8 Embedding 2}
\label{case8emb2}
\end{minipage}
\end{figure}

From these cases, we can conclude that no splitting of the degree 5 vertex will create a minor-minimal IPL graph with 16 edges. Thus, we can conclude none of the graphs containing $P_8$ as a minor on 16 edges that result from vertex splittings are minor-minimally IPL.

Now, consider graph \textbf{$P_{9}$(F)}.
This graph has six vertices of degree 3 and three vertices of degree 4. We only have to consider the degree 4 vertices. They are all equivalent, so without loss of generality, we have chosen vertex 5. Assign numbers to its adjacent edges: edge 1 connects vertex 5 to vertex 4, edge 2 connects vertex 5 to vertex 8, edge 3 connects vertex 5 to vertex 6, and edge 4 connects vertex 5 to vertex 2.
\begin{figure}[H]
\centering
\includegraphics[scale=0.62]{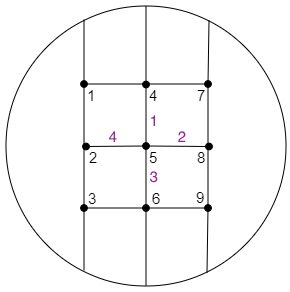}
\caption{Graph $P_9$ with the edges incident to vertex 5 labelled.}
\end{figure}

Case 1: Suppose three of the four edges connect to the new vertex. This will create a homeomorphic graph, so this case will create a projective planar graph.

Case 2: Suppose two edges that bound the same region connect to the new vertex. This case will create a projective planar graph, as seen in Figure \ref{p9case2}.

Case 3: Suppose opposite vertices, each from a degree 4 vertex, connect to the new vertex. For this example, this would mean edges 1 and 3 connect to the new vertex. This first embedding is not projective planar, as seen in Figure \ref{case3p9emb1}. However, there is another embedding that is projective planar, as seen in Figure \ref{case3p9emb2}, so this case does not create an IPL graph.
\begin{figure}[H]
\centering
\begin{minipage}[b]{0.3\linewidth}
\centering
\includegraphics[scale=0.62]{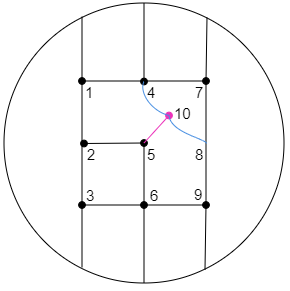}
\caption{Case 2 Embedding}
\label{p9case2}
\end{minipage}
\quad
\begin{minipage}[b]{0.32\linewidth}
\centering
\includegraphics[scale=0.63]{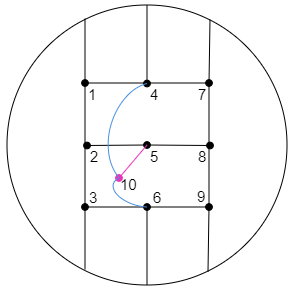}
\caption{Case 3 Embedding 1}
\label{case3p9emb1}
\end{minipage}
\quad
\begin{minipage}[b]{0.32\linewidth}
\centering
\includegraphics[scale=0.63]{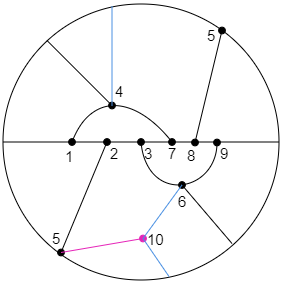}
\caption{Case 3 Embedding 2}
\label{case3p9emb2}
\end{minipage}
\end{figure}

Case 4: Suppose opposite vertices, each from a degree 3 vertex, connect to the new vertex. For this example, this would mean edges 2 and 4 connect to the new vertex. This case is equivalent to Case 3, except the new vertex and vertex 5 are switched - the graphs are isomorphic. Therefore, this graph will also be projective planar.

We can conclude that no splitting of the degree 4 vertices of $P_{9}$ will create a minor-minimal IPL graph with 16 edges. Thus, we can conclude none of the graphs containing $P_{9}$ as a minor on 16 edges that result from vertex splittings are minor-minimally IPL.

Finally, consider graph \textbf{$P_{10}$(G)}.
For $P_{10}$, every vertex has degree 3. Every vertex splitting will be projective planar because it will be homeomorphic to the original graph. Thus, we can conclude none of the graphs containing $P_{10}$ as a minor on 16 edges that result from vertex splittings are minor-minimally IPL.

\subsection{35 minor-minimal nonprojective planar graphs}


Only one of the 35 minor-minimal nonprojective-planar graphs is IPL, and it is $K_{4.4} -e$. A direction for future research would be to examine those graphs plus $K_1$. We conjecture that the resulting graphs would all be IPL.

\section{Acknowledgements}
This material is based upon work obtained by the research group at the 2021 Reserach Experience for Undergraduates Program at SUNY Potsdam and Clarkson University, advised by Joel Foisy and supported by the National Science Foundation under Grant No. H98230-21-1-0336; St. Catharine's College, University of Cambridge; and Universidad Autónoma del Estado de Hidalgo. We would like to thank Jeffrey Dinitz for Archdeacon’s list of minor-minimal nonouter-projective-planar graphs.

\end{document}